\documentclass[a4paper,11pt,english]{article}

\usepackage[utf8]{inputenc}
\usepackage[T1]{fontenc}
\usepackage{babel}
\usepackage{fullpage}
\usepackage{amsfonts,dsfont,amsmath,amsthm,amssymb,amsbsy,upgreek}
\usepackage{graphicx,color}
\usepackage[all]{xy}
\usepackage{tikz}
\usetikzlibrary{matrix,decorations.pathreplacing, calc, positioning}
\usepackage[normalem]{ulem}
\usepackage{enumerate}
\usepackage{multirow}
\usepackage{units}
\usepackage{multirow}
\usepackage{nccmath}
\usepackage{url}
\usepackage{cite}
\usepackage[pdftex,colorlinks=true,pagebackref=true
,linkcolor=blue,bookmarks=true,bookmarksopen=true,citecolor=blue]{hyperref}
\usepackage{titletoc}
\usepackage{mathptmx}
\newtheorem{theo}{Theorem}[section]
\newtheorem{defi}{Definition}[section]
\newtheorem{cor}{Corollary}[section]
\newtheorem{lem}{Lemma}[section]
\newtheorem{prop}{Proposition}[section]

\newtheorem{rem}{Remark}[section]

\newtheorem{ass}{Assumption}[section]
\newtheorem*{ack}{Acknowledgement}

\newcommand{\Nset}{\mathbb{N}}

\newcommand{\PP}{\mbox{$\mathbb{P}$}}

\newcommand{\CA}{\mathcal{A}}

\makeatletter
\newcommand{\xRightarrow}[2][]{\ext@arrow 0359\Rightarrowfill@{#1}{#2}}
\makeatother

\newcommand{\longhookrightarrow}{}
\DeclareRobustCommand{\longhookrightarrow}{\lhook\joinrel\relbar\joinrel\rightarrow}

\newcommand{\llpref}{\overleftarrow{\rm pref}\,}

\newcommand{\ttu}{\mathtt u}
\newcommand{\ttd}{\mathtt d}
\newcommand{\ttN}{\mathtt n}
\newcommand{\ttS}{\mathtt s}
\newcommand{\ttE}{\mathtt e}
\newcommand{\ttW}{\mathtt w}
\newcommand{\ttL}{\ell}

\newcommand{\tail}{\mathcal{T}}

\usepackage[draft]{fixme}
\fxusetheme{color}
\FXRegisterAuthor{a}{aa}{A}
\FXRegisterAuthor{pe}{ape}{P}
\FXRegisterAuthor{yo}{ayo}{Yo}
\FXRegisterAuthor{b}{ab}{B}

\usepackage[blocks]{authblk}

\title{\bf Recurrence of Multidimensional Persistent Random Walks. Fourier and Series Criteria}

\author[1]{{\bf Peggy Cénac}}
\author[2]{{\bf Basile de Loynes}}
\author[1]{{\bf Yoann Offret}}
\author[1]{{\bf Arnaud Rousselle}}

\affil[1]{Institut de Mathématiques de Bourgogne (IMB) - UMR CNRS 5584\\

Université de Bourgogne Franche-Comté, 21000 Dijon, France}

\affil[2]{
Ecole Nationale de la Statistique et de l’Analyse de l’Information (ENSAI)\\

Campus de Ker-Lann, rue Blaise Pascal, BP 37203, 35172 Bruz cedex, France}

\date{}

\begin{document}


\maketitle


\noindent
{\bf Abstract} 
The recurrence features of persistent random walks built from variable length Markov chains are investigated. We observe that these stochastic processes can be seen as Lévy walks for which the persistence times depend on some internal Markov chain:  they admit Markov random walk skeletons. A recurrence versus transience dichotomy is highlighted. We first give a sufficient Fourier criterion for the recurrence, close to the usual Chung-Fuchs one,  assuming in addition  the positive recurrence of the driving chain and a series criterion is derived. The key tool is the Nagaev-Guivarc'h method. Finally, we focus on particular two-dimensional persistent random walks, including directionally reinforced random walks, for which necessary and sufficient Fourier and series criteria are obtained.  Inspired by \cite{Rainer2007}, we produce a genuine counterexample to the conjecture of \cite{Mauldin1996}. As for the one-dimensional situation studied in  \cite{PRWI}, it is easier for a persistent random walk than its skeleton to be recurrent but here the difference is extremely thin. These results are based on a surprisingly novel -- to our knowledge -- upper bound  for the Lévy concentration function associated with symmetric distributions.\\


\noindent
{\bf Key words} {Persistent random walks {\bf .\@} Variable length Markov Chains {\bf .\@} Markov random walks {\bf .\@} Fourier and series recurrence criteria  {\bf .\@} Fourier perturbations {\bf .\@}  Markov operators   {\bf .\@} Concentration functions}\\


\noindent
{\bf Mathematics Subject Classification (2000)} {60J15 {\bf .\@} 60G17 {\bf .\@} 60J05 {\bf .\@} 37B20 {\bf .\@} 60K99  {\bf .\@} 60E05 {\bf .\@} 47A55 
} 


{\small
\tableofcontents
}


\section{Introduction}

Classical random walks are usually defined from a sequence of independent and identically distributed {\it i.i.d.\@} increments $\{X_k\}_{k\geq 1}$ by $S_0=0$ and for every $n\geq 1$,
\begin{equation}
\label{def-persist-part}
S_n:=\displaystyle\sum_{k=1}^n X_k.
\end{equation}

In  the continuity of \cite{PRWI} we aim at investigating  the asymptotic behaviour, and more specifically the recurrence features, of a multidimensional Persistent Random Walk (PRW) for which the increments are driven by a Variable Length Markov Chain (VLMC)  built from some probabilized context tree. This construction furnishes a wide class of models for the dependence of the increments which can be easily adapted to various situations. The toy model in \cite{PRWI} corresponds to the case of a VLMC built from a double-infinite comb and increments belonging to $\{-1,1\}\subset\mathbb Z$.

The characterization of the recurrent versus transient behaviour is difficult for a general probabilized context tree (see \cite{ccpp} for some zoology for instance). Before investigating a larger class of models,  we focus in Section \ref{quadruple} on a particular context tree  generalizing in $\mathbb Z^2$ the double-infinite comb already studied. The latter, naturally called a quadruple-infinite comb -- likewise, the resulting PRW is called the quadruple-infinite comb PRW -- is mainly motivated by two reasons. 

From this particular case, we point out in Section \ref{alphagis}  that such PRW can be seen -- rather generically -- as a continuous-time Markov Random Walks (MRW) called in the sequel a Markov Lévy Walk (MLW). This representation has motivated our will to extend the recurrence and transience criteria to this largest and worthwhile class of persistent stochastic processes.

Besides, another motivation was to answer to the conjecture \cite[Section 3., p.247]{Mauldin1996} related to  Directionally Reinforced Random Walks (DRRWs) in $\mathbb Z^2$. Those are in particular quadruple-infinite comb  PRWs for which the {\it i.i.d.\@} waiting times in $\{1,2,\cdots\}$ do not depend on some internal Markov chain  and the successive directions (four possibilities) are chosen uniformly among all excepted the last one (thus three uniform choices). The authors in \cite{Rainer2007} have partially answered by the negative to this guess and this question is definitively closed in this paper.


\subsection{The quadruple-infinite comb model}

\label{quadruple}

Let us start with the general construction of VLMCs built from a probabilized context tree  on the  alphabet $\mathcal{A}:=\{\ttE,\ttN,\ttW, \ttS\}$. In the sequel, we associate with every $\ell\in\mathcal A$ the corresponding direction in $\mathbb Z^2$ in such a way that $(\overrightarrow{\mathtt e},\overrightarrow{\mathtt n})$ stands for the canonical basis whereas $(\overrightarrow{\mathtt w},\overrightarrow{\mathtt s})$ is the opposite one. Hence, the letters $\ttE$, $\ttN$, $\ttW$ and $\ttS$ will stand for moves to the east, north, west and south respectively.

Let $\mathcal{L}= \CA^{-\Nset}$ be the set of left-infinite words  and  consider a complete tree on $\mathcal{A}$: each node has $0$ or ${\rm card}(\mathcal A)$ children. The set of leaves is denoted by $\mathcal{C}$ and elements of $\mathcal C$ are (possibly infinite) words on $\mathcal A$. To each leaf $c\in\mathcal C$, called a context, is attached a  distribution $q_{c}$ on $\mathcal A$. Endowed with this probabilistic structure, such a tree is named a probabilized context tree. The related VLMC -- here denoted by $\{U_{n}\}_{n\geq 0}$ -- is the Markov Chain on $\mathcal{L}$ whose transitions are given  by 
\begin{equation}
 \label{eq:def:VLMC}
 \PP(U_{n+1} = U_n\ell| U_n)=q_{\footnotesize\llpref (U_n)}(\ell),
 \end{equation}
where $\llpref(w)\in\mathcal C$ is defined as the shortest prefix of $w=\cdots w_{-1}w_{0}$, read from right to left, appearing as a leaf of the context tree. The $k$th increment $X_k$ of the corresponding PRW is identified with the rightmost letter of $U_k$. In particular, we can write 
\begin{equation}
U_{n}=\cdots X_{n-1}X_{n}.
\end{equation}

The set of leaves of the quadruple-infinite comb encodes the memory of the VLMC and consists of words on the alphabet $\mathcal A$ of the form
\begin{equation}\label{leaf}
\mathcal{C}:=\left\{\ell^n\ell^{\prime} : \ell\neq \ell^{\prime}\in \mathcal A,\; n\ge 1\right\}\cup \{\ell^\infty : \ell\in \mathcal A\}.
\end{equation} 
The prefix function is then formally defined by  
\begin{equation}
 \llpref (\cdots \ell^{\prime}\ell^{n})=\ell^n\ell^{\prime}\quad\mbox{and}\quad \llpref (\ell^{\infty})=\ell^{\infty}.
\end{equation}
One can summarize  the probabilistic structure as follows: for  $n\geq 1$, $\ttL,\ttL^{\prime},\ttL^{\prime\prime}\in\mathcal A$ with $\ell^\prime\neq \ell$ and $\ell^{\prime\prime}\neq \ell$, introduce $\alpha_n({\ttL^{\prime},\ttL})$ and $p_{n}((\ttL^{\prime},\ttL);(\ttL,\ttL^{\prime\prime}))$ so that 
\begin{equation}
\label{eq:tree}
q_{\ttL^n \ttL^{\prime}}(\ttL)  = 1-\alpha_n({\ttL^{\prime},\ttL})\quad\mbox{and}\quad
q_{\ttL^n \ttL^{\prime}}(\ttL^{\prime\prime})
=\alpha_{n}(\ttL^{\prime},\ttL)\, p_{n}((\ttL^{\prime},\ttL);(\ttL,\ttL^{\prime\prime})).	
\end{equation}
Also introduce $\alpha_{\infty}(\ttL,\ell)$ and $p_{\infty}((\ttL,\ttL);(\ttL,\ttL^{\prime\prime}))$ with
\begin{equation}\label{eq:tree2}
q_{\ell^\infty}(\ell)=:1-\alpha_{\infty}(\ttL,\ell)\quad\mbox{and}\quad{}
q_{\ell^\infty}(\ell^{\prime\prime})=:\alpha_{\infty}(\ttL,\ell)\,p_{\infty}((\ttL,\ttL);(\ttL,\ttL^{\prime\prime})).
\end{equation}
These quantities are interpreted as transition probabilities -- see Figure \ref{diag} -- and characterize the probabilized context tree.

\begin{figure}[!h]
\centering
\begin{tikzpicture}
\tikzstyle{local}=[rectangle,draw,rounded corners=4pt]
\node[local] (etat00) at (-6,0) {$\ttL^\infty$};
\node[local] (etat01) at (-2.5,1.25) {$\ttL^\infty$};
\node[local] (etat02) at (-2.5,-1.25) {$\ttL^\infty\ell^{\prime\prime}$};
\draw[->,>=latex] (etat00.east) -- (etat01.west) node[midway,fill=white]{\footnotesize $1-\alpha_\infty(\ell,\ell)$};
\draw[->,>=latex] (etat00.east) -- (etat02.west) node[midway,fill=white]{\footnotesize $\alpha_\infty(\ell,\ell)p_\infty((\ell,\ell);(\ell,\ell^{\prime\prime})$};

\node[local] (etat0) at (-1,0) {$\cdots \ttL^\prime \ttL^n$};
\node[local] (etat1) at (4,1.25) {$\cdots\ell^\prime\ell^{n+1}$};
\node[local] (etat2) at (4,-1.25) {$\cdots\ell^\prime\ell^{n}\ell^{\prime\prime}$};

\draw[->,>=latex] (etat0.east) -- (etat1.west) node[midway,fill=white]{\footnotesize $1-\alpha_n(\ell^\prime,\ell)$};
\draw[->,>=latex] (etat0.east) -- (etat2.west) node[midway,fill=white]{\footnotesize $\alpha_n(\ell^\prime,\ell)p_n((\ell^\prime,\ell);(\ell,\ell^{\prime\prime})$};
\end{tikzpicture}
\caption{\label{diag} Transitions probabilities of the quadruple-infinite comb}
\end{figure}
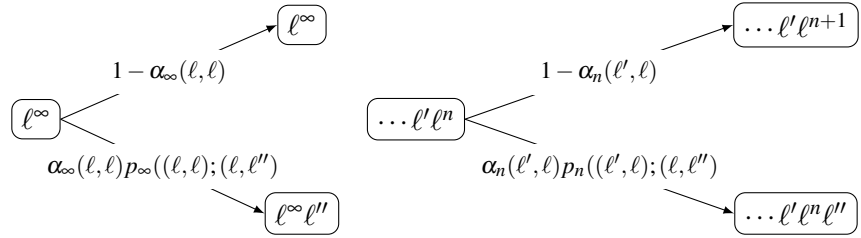

\begin{rem}
 Note that the probability for a change of direction depends on the time spent in the current direction but also, contrary to the one-dimensional PRWs, on the previous direction.
\end{rem}


\begin{figure}[!b]
\begin{center}
\includegraphics[width=88mm]{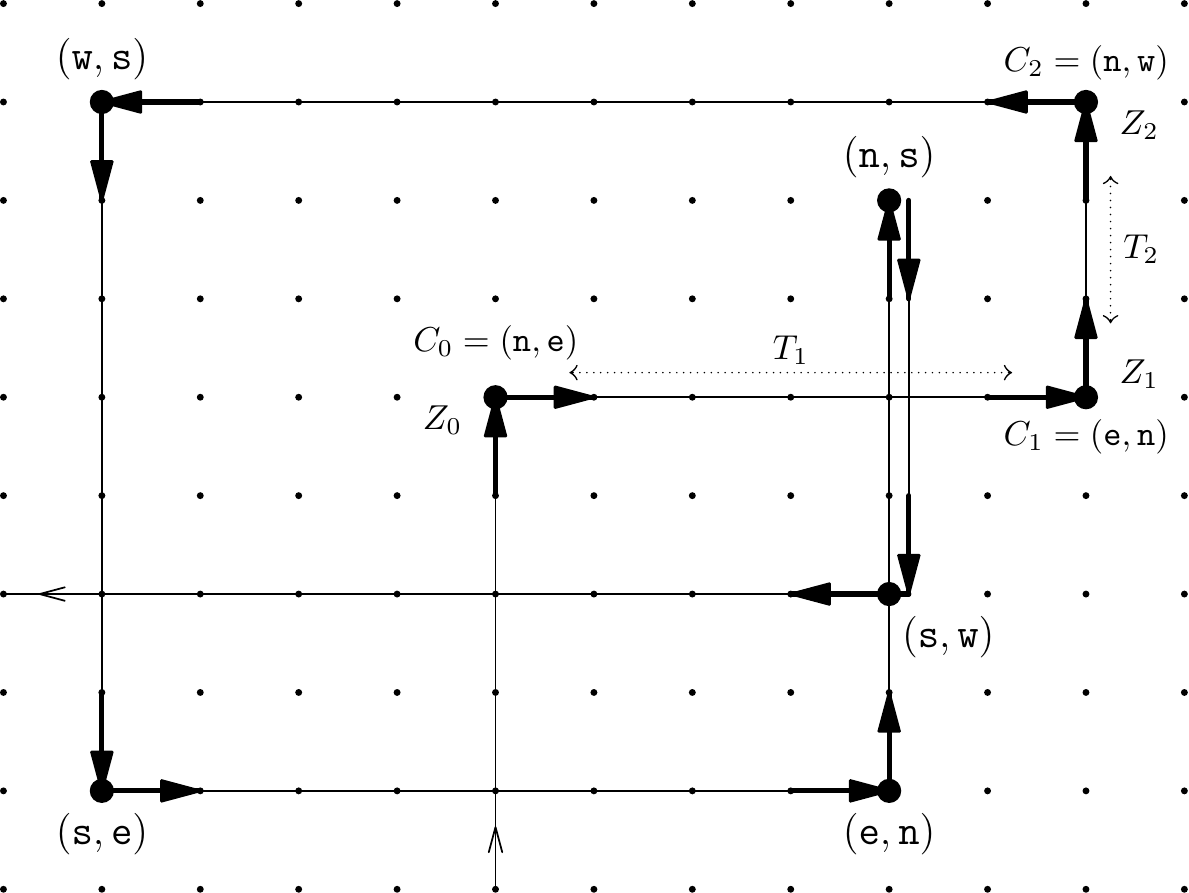}
\end{center}
\caption{Piecewise interpolation of the PRW}
\label{marche}
\end{figure}

In the following, we refer carefully to Figure \ref{marche} below that illustrates our notations and assumptions by a realization of a linear interpolation $\{S_t\}_{t\geq 0}$ of a quadruple-infinite comb PRW.


\subsubsection{Associated MRW}

\label{mrw}

Let $P$ be the Markov kernel on $\mathcal A\times \mathcal A$  defined for every $\ell^{\prime},\ell,\ell^{\prime\prime}\in\mathcal A$ with $\ell^{\prime}\neq\ell$ and $\ell^{\prime\prime}\neq \ell$ by
\begin{equation}\label{markovsymb1}
P((\ell^{\prime},\ell);(\ell,\ell^{\prime\prime})):=\sum_{n=1}^{\infty}\left(\prod_{k=1}^{n-1}(1-\alpha_{k}(\ell^{\prime},\ell))\right)\alpha_{n}(\ell^{\prime},\ell)p_{n}((\ell^{\prime},\ell);(\ell,\ell^{\prime\prime})),
\end{equation}
and
\begin{equation}\label{markovsymb2}
P((\ell,\ell);(\ell,\ell^{\prime\prime})):= \sum_{n=1}^{\infty}(1-\alpha_{\infty}(\ell^{},\ell))^{n-1}\alpha_{\infty}(\ell^{},\ell)p_{\infty}((\ell^{},\ell);(\ell,\ell^{\prime\prime})).	
\end{equation}
To get a stochastic matrix, we choose adequately the entries $P((\ell^{\prime},\ell);(\ell,\ell))$ and $P((\ell,\ell);(\ell,\ell))$ when necessary  and set to zero any others. For the sake of simplicity, it is supposed in the sequel the following assumption. 


\begin{ass}[]\label{A1} 
One has $(X_{0},X_{1})=(\ttN,\ttE)$ with probability one and the state $(\ttN,\ttE)$ belongs to an irreducible class $\mathcal S\subset \mathcal A\times\mathcal A\setminus \Delta$ of the Markov kernel $P$ where $\Delta\subset \mathcal A\times \mathcal A$ is the diagonal subset.
\end{ass}

Obviously, regarding the study of the asymptotic behaviour of the PRW, there is no loss of generality assuming such conditions. Note that, under this assumption,  for every $c\in \mathcal S$, 
\begin{equation}\label{ainit3}
 \sum_{n=1}^{\infty}\left(\prod_{k=1}^{n-1}(1-\alpha_{k}(c))\right)\alpha_{n}(c)=1.
\end{equation}
Roughly speaking, this assumption disallows a too strong reinforcement, that is  a too fast decreasing rate for the transition probabilities $\alpha_{n}(c)$ of changing directions. As a matter of fact, the transition probabilities between two changes of letters -- named  breaking or moving times -- are encoded by the Markov kernel $P$. In fact,  let $\{B_n\}_{n\geq 0}$ be the almost surely finite breaking times defined inductively by 
\begin{equation}
B_0=0\quad\mbox{and}\quad
B_{n+1}=\inf\left\{k>B_{n} : X_{k}\neq X_{k+1}\right\}.
\end{equation}
It turns out that the so called internal (configuration or driven) chain $\{C_n\}_{n\geq 0}$ defined by 
\begin{equation}\label{drivingchain}
C_{n}:=(X_{B_{n}},X_{B_{n+1}}),
\end{equation}
is an irreducible Markov chain  on  $\mathcal S$ starting from $(\ttN,\ttE)$ whose Markov kernel -- still denoted by $P$ abusing notation -- is the restriction of $P$ to $\mathcal S\times \mathcal S$.

The waiting times $T_{n+1}:=B_{n+1}-B_n$ are not independent contrary to the one-dimensional case. However, conditionally to the events $\{C_{n}=c,C_{n+1}=s\}$, they share the same distribution. The skeleton random walk -- the  PRW observed at the breaking times -- $\{Z_{n}\}_{n\geq 0}$ on $\mathbb Z^2$ is then defined as 
\begin{equation}
Z_n:=S_{B_n}=\sum_{i=1}^{n}\left(\sum_{k=T_{i-1}+1}^{T_{i}} X_k\right),
\end{equation}
where $T_0=0$. Obviously, $Z$ is not a RW. Nevertheless, taking into account the additional information given by the internal Markov chain, $Z$ is rather a  MRW,  also named  a Markov Additive Process (MAP), semi-Markov process or hidden Markov chain (see  \cite{Alsmeyer,Barbu2008} for instance). To be more specific, it means the process $\{(Z_n,C_n)\}_{n\geq 0}$ is Markovian on $\mathbb Z^2\times \mathcal S$ and satisfies 
\begin{equation}\label{MAPTRANS}
\mathcal L\big((Z_{n+1}-Z_n,C_n) \mid \{(Z_k,C_k)\}_{0\leq k\leq n}\big)=\mathcal L\big((Z_{n+1}-Z_n,C_n) \mid C_n\big).
\end{equation}
Here the latter conditional distributions do not depend on $n\geq 0$.

Thereafter, we  introduce  for every $c,s\in \mathcal S$ such that $P(c,s)>0$ the conditional jump and waiting time distributions
\begin{equation}
\mu_{c,s}(dx):=\mathbb P\left(Z_{n+1}-Z_n \in dx \,|\,C_n=c, C_{n+1}=s\right)\quad\mbox{and}\quad \mu_c(dt):=\sum_{s\in\mathcal S}P(c,s)\mu_{c,s}(dx),
\end{equation} 
and 
\begin{equation}\label{persist}
\nu_{c,s}(dt)=\mathbb P(T_{n+1}\in dt\mid C_n=c,C_{n+1}=s)\quad\mbox{and}\quad 
\nu_{c}(dt)=\sum_{s\in\mathcal S}P(c,s)\nu_{c,s}(dt).
\end{equation}
Here, writing a configuration $c$ as $(c_{\mathtt i},c_{\mathtt o})$, we get  $ \mu_{c,s}(n\overrightarrow {c_{\mathtt o}})=\nu_{c,s}(n)$ and $\mu_c(n\overrightarrow {c_{\mathtt o}})=\nu_c(n)$, the latter denoting the distribution of the $n$th term of the summand in (\ref{ainit3}) and the former equals to 
\begin{equation}
\label{loi-cond}
\frac{p_{n}(c;s) \alpha_{n}(c)\prod_{k=1}^{n-1}(1-\alpha_{k}(c))}{P(c,s)}.
\end{equation}
Hence, $\mu_{c,s}(dx)$ and $\mu_{c}(dx)$ can be viewed as conditional vectorial persistence time distributions in which are simultaneously encoded the length and the direction.


\subsubsection{A generic MLW structure for PRWs}

\label{alphagis}

\label{MLW}

Therefore, at the sight of the considerations above,  a PRW can be constructed as follows:

\begin{itemize}
\item introduce a Markov chain  $\{C_n\}_{n\geq 0}$  on $\mathcal S$ with transition kernel $P$; 
\item consider independent sequences of {\it i.i.d.\@} random variables $\{(\tau_{n}(c,s),\overrightarrow{\tau_{n}}(c,s))\}_{n\geq 1}$ -- themselves independent of $\{C_n\}_{n\geq 0}$  -- taking  values in $\{1,2,\cdots\}\times{\mathbb Z}^2$ for every $c,s\in\mathcal S$ and  whose respective distributions are the push-forward images of $\mu_{c,d}$ by $x\longmapsto (\|x\|,x)$;
\item the piecewise linear interpolation of the initial discrete-time PRW is then given by  
\begin{equation}\label{PRW1}
S_t:=\sum_{n=1}^{N(t)}\overrightarrow{\tau_n}(C_{n-1},C_n)+\left(t-B_{N(t)}\right)\overrightarrow{\tau_{N(t)+1}}(C_{N(t)},C_{N(t)+1}),
\end{equation}
with
\begin{equation}\label{PRW2}
N(t):=\max\left\{n\geq 0 : B_n\leq t\right\}\quad\mbox{and}\quad B_n:=\tau_1(C_0,C_1)+\cdots+\tau_n(C_{n-1},C_n);
\end{equation}
\item and the  skeleton MRW is obtained setting 
\begin{equation}\label{ctrw}
Z_n:=S_{B_n}=\sum_{k=1}^{n}\overrightarrow{\tau_k}(C_{k-1},C_k).
\end{equation}

\end{itemize}

Following the terminology of  \cite{Magdziarz}, the continuous-time persistent process $\{S_t\}_{t\geq 0}$ is virtually a Lévy Walks (LW), except that the waiting times, as well as the jumps, are no longer {\it i.i.d.\@} nor even independent. Original LWs are basically Continuous Time Random Walks (CTRWs) -- the summand in (\ref{PRW1}) when the distributions $\mu_{c,s}$ do not depend on $c,s\in\mathcal S$ -- for which the waiting times and the sizes of jumps are coupled and usually proportional. In our context, the continuous interpolation  (\ref{PRW1}) of the skeleton MRW is called a Markov Lévy Walk (MLW) to fit the denomination of the skeleton MRW and the LW structure. As well explained in \cite{Magdziarz} and also in \cite{Meer1,Meer2,MeerStra,Straka}, these kind of stochastic processes model a wide panoply of phenomena involving anomalous diffusions.

Obviously, there are many other possible choices for the internal chain, all of them leading to different cutting of the trajectories of the original PRW. For instance, remarking that a PRW is an additive functional of the underlying VLMC, one may choose for internal chain the VLMC itself. The jump distributions are then deterministic. With this choice, the geometry of the ambient space and the symmetries are forgotten. Somehow, the valuable information is entirely encoded in the internal chain, \emph{i.e.\@} the VLMC. In fact, it would be wise to find  a tradeoff between the complexity of the internal chain and that of the jump distributions.

Actually, the choice intuitively made so far for the double or the quadruple-infinite comb PRW can be generically achieve using the notion of Greatest Internal Suffix (GIS) defined in \cite{CCPP3}. Keeping the notation introduced in Sections \ref{quadruple}, consider $\theta$ the left-shift operator defined  for every (possibly infinite) words  $\omega:=\omega_1\omega_2\cdots$ on the alphabet $\mathcal A$  by $\theta(\omega)=\omega_2\omega_3\cdots$ and set for every context $c\in\mathcal C$,
\begin{equation}
{\rm \alpha gis}(c):=\theta^{\tau(c)-1}(c)\in\mathcal C,\quad\mbox{with}\quad \tau(c):=\inf\{ n\geq 1 : \theta^{n}(c)\notin\mathcal C\}.
\end{equation} 
The word ${\rm \alpha gis}(c)$ is said to be the $\alpha$-GIS associated with $c$ in the sense that ${\rm gis}(c):=\theta({\rm \alpha gis}(c))$ is the longest suffix appearing as an internal node of the context tree. We denote by $\mathcal G \subset\mathcal C$ the set of $\alpha$-GISs. The latter is a good candidate for the  internal state space.  For this choice, we retrieve for instance $\{\ttu\ttd,\ttd\ttu\}$ and $\{\mathtt e,\mathtt n,\mathtt w,\mathtt s\}^2$ for the double and  the quadruple-infinite comb models respectively. 

To go further, define inductively the sequence of breaking times as follows 
\begin{equation}
B_0=0\quad\mbox{and}\quad B_{n+1}:=\inf\big\{k>B_n : {\rm \alpha gis}(\llpref U_{k})\neq {\rm \alpha gis}(\llpref U_{B_n})\big\},
\end{equation}
and set as previously $T_{n+1}:=B_{n+1}-B_n$ and $T_0=0$ but also
\begin{equation}
C_n:={\rm gis}(\llpref U_{B_n})\quad\mbox{and}\quad  Z_n:=S_{B_n}=\sum_{i=1}^n\left(\sum_{k=T_{i-1}+1}^{T_i}X_k\right).
\end{equation}
Then assuming $\llpref U_{0}\in\mathcal G$, it turns out that $\{(Z_n,C_n)\}_{n\geq 0}$ is a MRW skeleton of the PRW and the latter  can be recovered adding the information given by the conditional excursions
\begin{equation}
{\mathbf e_{g,h}}(d\xi):=\mathbb P\left(n\longmapsto\sum_{k=T_{i-1}+1}^{n\wedge T_{i}} X_k \in d\xi\Bigg| C_0=g,C_1=h \right).
\end{equation}

\begin{rem}
  There is no reason for a context tree to admit a finite set of GISs. That is why, in the sequel, internal Markov chains evolving in a possibly infinite countable state space are considered. It is worth noting these considerations are not artificial: consider for instance a one-dimensional PRW with increments in $\{ -1,1 \}$ whose memory is encoded through the length of last rise together with the length of the last descent.

\end{rem}


\subsection{Overview of the article}


Foremost, note that in Section \ref{Results}  is considered a general  MLW on $\mathbb R^d$, $d\geq 1$. Such processes are easily defined adapting slightly the construction in Section \ref{MLW}. 

In Section \ref{dichotomy}, it is first proved that the MLW, as well as its embedded skeleton MRW, are either recurrent or transient supposing the internal Markov chain is recurrent (Proposition \ref{dicho}). If in addition the internal Markov chain is supposed positive recurrent, then it is shown that $Z$ is recurrent if and only some series is infinite as for classical RWs (Proposition \ref{dicho2}). This characterization consists in extending a result of \cite{Alsmeyer} to multidimensional MRWs.

In Theorem~\ref{Zcriterion} of Section~\ref{secFourier} are stated Fourier and Series criteria characterizing the type (recurrent or transient) of the skeleton MRW.  Eventhough, the proof of this result basically follows the ideas of the Nagaev-Guivarc'h perturbation method, it is worth noting that no moment conditions are assumed so that virtually all kind of jump distributions can be considered. Indeed, for our purpose, we consider situations where there is no probability invariant measure for the VLMC $(U_n)$ (see for instance \cite{peggy}). 

Some probabilistic and operator Assumptions \ref{assprob} and \ref{assop} together with some Sector Assumption \ref{sector} are obviously required and mostly relevant in the case of an infinite internal state space. The analytic criterion (\ref{fourier}) is in a first approximation nothing but the classical Chung-Fuch criterion for some averaged, classical  random walk (Remark \ref{approx}). The usual characteristic function is replaced by the principal eigenvalue of some Fourier perturbation associated with the internal Markov operator. This principal eigenvalue admits the series expansion \eqref{expans}. Because of the different nature of Fourier analysis in the lattice and non lattice cases, this section only deals with MRW taking values in $\mathbb Z^d$.

Unfortunately, in view of \cite[Theorem 4., p.  684]{Rainer2007}, Theorem \ref{Zcriterion} only gives a sufficient criterion for the recurrence of a MLW. The result in \cite{Rainer2007} also answers nearly by the negative to the conjecture about two-dimensional DRRWs in \cite[Section 3., p.247]{Mauldin1996}. Informally, it is asked whether a DRRW is recurrent simultaneously with the RW defined as the DRRW observed at the successive times of returns in its initial direction. As already pointed out by the authors, the given  example in \cite{Rainer2007} do no fit well to the usual framework of DRRW  since their waiting times can be equal to zero with a positive probability. This may appear anecdotal, however,  their ingenious and technical construction involves in a crucial way unimodality arguments that can not be applied for true DRRWs. Nevertheless, it still provides a counter-example for our general MLWs and related MRW skeletons.

In Section \ref{criteriacomb}, returning to the quadruple-infinite comb model, a complete characterization of the recurrence of PRWs is stated in Proposition \ref{coro}. For this specific model, the margins of the resulting skeleton are independent symmetric one-dimensional RWs. The admissible probabilistic structure is detailed in Assumptions \ref{modelrestrein} and includes DRRWs.

\begin{rem}
In \cite[Theorem 2., p. 682]{Rainer2007}, it is proved that DRRWs are transient in $\mathbb Z^d$ for $d\geq 3$. As far as the type problem is concerned, the higher dimensional cases modeled on Assumption \ref{modelrestrein} seem to be irrelevant and is not investigated in this paper.
\end{rem}

Our results are based on the fundamental Lemma  \ref{lévysym} and Theorem \ref{critère} involving an appropriate Borel-Cantelli Lemma. Let us stress that, to our knowledge, the result in Lemma \ref{lévysym} is surprisingly not mentioned anywhere. At the end of this section, the conjecture \cite[Section 3., p.247]{Mauldin1996} for DRRWs, and actually for a wider class of two-dimensional PRWs, is definitively answered by the negative. We follow the  constructive probabilistic approach presented in \cite{Rainer2007}, excepting that the unimodality assumption is dropped requiring the important Lemma \ref{lévysym}.

Finally, the two remaining sections are devoted to the proofs of Propositions \ref{eigenvalue} and \ref{expansprop} and Theorem \ref{Zcriterion} of Section \ref{prooffourier}, the fundamental Lemma \ref{lévysym} and its consequences  in Theorem \ref{critère}, of Corollary \ref{coro} and Theorem \ref{conjecture} of Section \ref{proofcomb}.

\section{Recurrence and transience criteria}

\setcounter{equation}{0}

\label{Results}

In this section, we consider a Markov chain $\{C_n\}_{n\geq 0}$ on a discrete and countable state space $\mathcal S$ whose Markov kernel is denoted by $P$. Also, we denote by $\pi(dc)$ a corresponding invariant measure, normalized to be a probability when possible. Let $\{\tau_{n}(c,s),\overrightarrow{\tau_{n}}(c,s)\}_{n \geq 0}$  be a sequence of \emph{i.i.d.} random variables taking values in $[0,\infty)\times \mathbb R^d$ whose common distribution, depending only on $(c,s)$, is denoted by $m_{c,s}(dt,dx)$. Moreover, let us introduce the first and  second marginal distributions of $m_{c,s}$ denoted respectively by $\nu_{c,s}(dt)$ and $\mu_{c,s}(dx)$  and set 
\begin{equation}\label{conditionaljump}
\mu_c(dx):=\sum_{s\in\mathcal S}P(c,s)\mu_{c,s}(dx)\quad\mbox{and}\quad \nu_c(dt):=\sum_{s\in\mathcal S}P(c,s)\mu_{c,s}(dt).
\end{equation}

Thereafter, one can construct as in (\ref{PRW1})-(\ref{ctrw}) above a MLW denoted by $\{S_t\}_{t\geq 0}$ evolving in $\mathbb R^d$ whose skeleton $\{Z_n\}_{n\geq 0}$ is a MRW coupled with $C$ as an internal Markov process. In order to ensure the continuity of $S$, it is assumed that, for every $c,s\in\mathcal S$ with $P(c,s)>0$,
\begin{equation}
m_{c,s}(\{0\}\times (\mathbb R^d\setminus \{0\}))=0.
\end{equation}

In the sequel $\mathbb P_c$ (\emph{resp.} $\mathbb P_\nu$) denotes the probability distribution on the path space conditionally to $C_0=c$ (\emph{resp.} $C_0$ is  distributed as $\nu$) and $S_0=Z_0=0$.


\subsection{Dichotomy results}

\label{dichotomy}

The following proposition states a zero-one law for MLWs and MRWs leading to the standard dichotomy between recurrence \emph{versus} transience. 

\begin{prop}[zero-one law and dichotomy recurrence/transience]\label{dicho} Assume that the internal Markov chain $C$ is irreducible and recurrent. Then, for any  $c\in\mathcal S$ and any Borel subset $A\subset \mathbb R^d$, it holds
\begin{equation}\label{asympevents}
\mathbb P_c\left(\bigcap_{t\geq 0}\bigcup_{u\geq t}\{S_u\in A\}\right)\in\{0,1\}\quad\mbox{and}\quad \mathbb P_c\left(\bigcap_{n\geq 0}\bigcup_{k\geq n}\{Z_k\in A\}\right)\in\{0,1\}.
\end{equation}
In particular, a MLW (\emph{resp.} MRW) is either recurrent or transient in the sense that either for any $c\in\mathcal S$,
\begin{equation}\label{rec2}
\mathbb P_{c}\left(\lim_{t\to\infty} \|S_{t}\|=\infty\right)=1\quad\left(resp.  \quad\mathbb P_{c}\left(\lim_{n\to\infty} \|Z_{n}\|=\infty\right)=1\right),\tag{T}
\end{equation}
or for any $c\in\mathcal S$ there exists $r>0$ (\emph{a priori} depending on $c$) such that
\begin{equation}\label{rec1}
\mathbb P_{c}\left(\liminf_{t\to\infty}\|S_{t}\|<r\right)=1 \tag{R-1}
\quad\left(\mbox{resp.}\quad\mathbb P_{c}\left(\liminf_{n\to\infty} \|Z_n\|<r\right)=1\right).
\end{equation}
\end{prop}

\begin{rem}
Considering the original problematic of PRWs built from VLMCs, an analogous zero-one law holds substituting  $\{S_t\}_{t\geq 0}$ with the discrete-time process $\{S_n\}_{n\geq 0}$.
\end{rem}

\begin{proof}
Let $c \in \mathcal S$ be and introduce the successive visit times $\{\sigma_{n}\}_{n\geq 0}$ of $c$. One define 
\begin{equation}
\mathbb X_n:=\left\{\overrightarrow{\tau_{\sigma_n+k}}(C_{\sigma_n+k-1},C_{\sigma_n+k})\right\}_{1\leq k\leq \sigma_{n+1}-\sigma_{n}},
\end{equation}
for all $n\geq 0$. Since the excursions between two visits of $c$ are {\it i.i.d.\@} under $\mathbb P_c$, so it is for $\{\mathbb X_n\}_{n\geq 0}$. Therefore, the zero-one law (\ref{asympevents}) follows from the Hewitt-Savage zero-one law \cite[Theorem 3.15, p.53]{Kal} noting that the asymptotic events belong to the exchangeable $\sigma$-field of $\{ \mathbb X \}_{n \geq 0}$.

Specifying the zero-one law \eqref{asympevents} to the events considered in \eqref{rec2} and \eqref{rec1} so that they occur with probability zero or one, it only remains to prove these probabilities do not depend on the initial configuration $c$. To this end, suppose that $S$ (\emph{resp.} $Z$) goes to infinity for one configuration $c$. Then, the irreducibility of $C$ and the translation invariance property \eqref{MAPTRANS} of Markov additive processes imply $S$ goes to infinity with a positive probability, and in turn with probability one, for any internal state.
\end{proof}

Assuming in addition $C$ is $\pi$-positive recurrent,  one can improve (\ref{rec1}) for MRWs. To this end, introduce the recurrent set $\mathcal R$ and the set of  possible points $\mathcal P$ defined by 
\begin{equation}
\mathcal R:=\left\{x\in\mathbb R^d : \forall \varepsilon>0,\; \mathbb P_\pi(Z_n\in B(x,\varepsilon)\;{\rm i.o.})=1\right\},
\end{equation}
and
\begin{equation}
\mathcal P:=\{x\in\mathbb R^d : \forall \varepsilon>0,\;\exists\,n\geq 0,\; \mathbb P_\pi(Z_n\in B(x,\varepsilon))>0\},
\end{equation} 
where $B(x,\varepsilon)\subset \mathbb R^d$ stands for the open ball of radius $\varepsilon>0$ centered at $x\in\mathbb R^d$. Note that $\mathcal R$ and $\mathcal P$ are both closed subsets. Now let $\Gamma \subset \mathbb R^d$ be the smallest closed subgroup containing the support of the distribution mixture
\begin{equation}\label{mixture}
\mu_\pi(dx):=\sum_{c\in\mathcal S}\pi(c)\mu_c(dx),
\end{equation}
where $\mu_c(dx)$ is given in (\ref{conditionaljump}).

\begin{prop}[recurrence features for MRWs and series criterion]\label{dicho2}
Assume that the internal Markov chain is irreducible and $\pi$-positive recurrent. Then one has $\mathcal P= \mathcal R=\Gamma$ when the MRW is recurrent. Furthermore,  the alternative (\ref{rec1})  is equivalent for the MRW to each of the following statements.
\begin{enumerate}
\item For some (or equivalently any) $\varepsilon>0$  and some (or any) initial distribution $\nu$,
\begin{equation}\label{rec3}
\mathbb P_{\nu}\left(\liminf_{n\to\infty} \|Z_n\|<\varepsilon\right)=1.\tag{R-2}
\end{equation}
\item For some (or any) $\varepsilon>0$ and some (or any) initial state $c\in\mathcal S$,
\begin{equation}\label{rec4}
\sum_{n=0}^\infty\mathbb P_c(Z_n\in B(0,\varepsilon))=\infty.\tag{R-3}
\end{equation}

\end{enumerate}
\end{prop}

\begin{rem}
Regarding the recurrence set associated with $S$ the question seems to be more intricate since it depends strongly on the geometry of each conditional jumps $\mu_{c,s}(dx)$. 
For instance, one can be easily convinced that it is possible for two recurrent PRWs to have both recurrent skeletons in $\mathbb Z^2$ but distinct recurrent set given respectively by $\mathbb R^2$ and
$\{(x,y)\in\mathbb R^2 : x\in\mathbb Z\mbox{ or } y\in\mathbb Z\}$.
\end{rem}

\begin{proof}
First, we deduce from the partition exhibited in \cite{berbeeRecTrans} for stationary random walks and from the zero-one law in Proposition \ref{dicho} that (\ref{rec1}) is equivalent to (\ref{rec3}) when $\nu=\pi$, and thus for any (or some) arbitrary $\nu$ since $\pi$ is fully supported. Besides, one can easily see that the relevant  Propositions in \cite[pp. 127-130]{Alsmeyer}  can be adapted to a multidimensional  framework. The first Proposition for instance which is originally  taken from \cite[p. 56]{berbeethesis} can be more generally obtained for multidimensional MRW using \cite{berbeeRecTrans} together with the dichotomy Proposition \ref{dicho}. It follows that the recurrent alternative is equivalent to (\ref{rec4}) and $\mathcal P=\mathcal R=\Gamma$.
\end{proof}

\subsection{A general sufficient Fourier criterion}

 \label{secFourier}

Fourier analysis is substantially different in the lattice and non lattice case. That is why, from now on, we make the choice to restrict ourself to MRWs taking values in $\mathbb Z^d$ which is irrelevant for the study of PRWs built from VLMCs. Besides, Fourier analysis in the lattice context usually requires a notion of aperiodicity defined below. 

\begin{defi}[aperiodic MRW]
A MRW is said to be  periodic if for some $c\in\mathcal S$, $x\in\mathbb Z^d$ and proper subgroup $\Gamma\subsetneq \mathbb Z^d$ it holds  $\mu_c(x+\Gamma)=1$. On the contrary, it is said  aperiodic.
\end{defi}

In the sequel, we make the following probabilistic assumptions.

\begin{ass}[probabilistic assumptions]~\label{assprob}
\begin{enumerate}
\item[\hypertarget{P1}{\rm (P1)}] The internal Markov chain $C$ is irreducible, aperiodic (classical sense) and $\pi$-positive recurrent.
\item[\hypertarget{P3}{\rm (P2)}] The Markov random walk $Z$ is  aperiodic in $\mathbb Z^d$. 

\end{enumerate}
\end{ass}

We  introduce for every $t\in\mathbb T^d$ -- the $d$-dimensional torus $\mathbb R^d/2\pi\mathbb Z^d$ -- the operator on $\mathbb{L}^1(\pi)$ defined for every  $f\in\mathbb{L}^1(\pi)$ and  $c\in\mathcal S$ by
\begin{equation}
P_tf(c):=\mathbb E_c[e^{itZ_1}f(C_1)].
\end{equation}
Regarding such Fourier perturbations  we refer to \cite{Gui:83,Bab:88,Uch:07,guilepagestable,HervePen} for instance. We recall that the peripheral spectrum -- the set of spectral values with maximal modulus -- is well defined for bounded  operators. Moreover, we say that a Markov operator has a spectral gap when its spectrum outside a centered ball of radius $1-\rho$ is finite for some $0<\rho<1$.

\begin{ass}[operator assumptions]\label{assop} There exists a Banach space $(\mathcal B,\|\cdot \|_{\mathcal B})$ such that:
\begin{enumerate}
\item[\hypertarget{O1}{\rm (O1)}] $\mathds 1\in \mathcal B$ and the canonical injection $\mathcal B\longhookrightarrow \mathbb{L}^1(\pi)$ is continuous; 
\item[\hypertarget{O2}{\rm (O2)}] the operators $P_t$ acts continuously on $\mathcal B$ for every $t\in\mathbb T^d$;
\begin{enumerate}
\item[\hypertarget{O2a}{\rm (a)}] the restricted Markov kernel $P : \mathcal B\longrightarrow \mathcal B$ admits a spectral gap; 
\item[\hypertarget{O2b}{\rm (b)}] the map $t\longmapsto P_t$ is continuous for the subordinated norm operator induced by $\|\cdot\|_{\mathcal B}$;
\item[\hypertarget{O2c}{\rm (c)}] the peripheral spectrum of $P_t: \mathcal B\longrightarrow \mathcal B$ only consists of eigenvalues.
\end{enumerate}
\end{enumerate}
\end{ass}

Remark those operator assumptions are satisfied for the Banach space $\mathbb{L}^2(\pi)$ when the $P_t$ are quasi-compact. We allude to \cite{KonMeyn,KonMeyn2,MeynTweedie,GarRosen,HerLedoux,Hennion,guibourg:hal}  for more general consideration about these properties but also -- when there exists Lyapunov functions -- for the interesting situation of weighted-supremum spaces  corresponding to geometric ergodicity. Before stating the last assumptions, let us draw some important consequences.

\begin{prop}[the eigenvalue of maximal modulus]\label{eigenvalue} Under Assumptions \ref{assprob} and \ref{assop}, for any sufficiently small neighbourhood $\mathcal V\subset\mathbb T^d$ of the origin and any $y \in \mathcal V$ the following properties hold:
\begin{enumerate}
\item the spectrum of $P_t$ admits a unique element of maximal modulus $\lambda(t)$; 
\item $\lambda(t)$ is an eigenvalue of algebraic multiplicity one;
\item the map $t\longmapsto\lambda(t)$ is continuous;
\item $|\lambda(t)|\leq 1$ and $|\lambda(t)|=1$  if and only if  $t=0$.
\end{enumerate}
\end{prop}

In particular, we can write the Markov operator $P$ as $Q+(P-Q)$ where $Q=\pi\otimes \mathds 1$ is the eigenprojector  on ${\rm span}(\mathds 1)$ defined by $Qf=(\pi f) \mathds 1$. Note that $P-Q$ leaves invariant ${\rm ker}(Q)={\Im}(1-Q)$ and thus commutes with $Q$. Introduce the linear bounded operator  $T : \mathcal B\longrightarrow\mathcal B$ given by 
\begin{equation}\label{reducedresolvent}
Tf:=\sum_{n=0}^\infty P^n(f-(\pi f)\mathds 1).
\end{equation}
This operator naturally appears in the series expansion of $\lambda(t)$ near the origin.

\begin{prop}[series expansion]\label{expansprop}
Under Assumptions \ref{assprob} and \ref{assop}, one has for any sufficiently small neighbourhood $\mathcal V\subset\mathbb T^d$ of the origin and any $t \in \mathcal V$ the expansion 
\begin{equation}\label{expans}
\lambda(t)-1=\sum_{n=0}^\infty (-1)^n\pi((P_t-P)T)^n(P_t-P)\mathds 1=\pi(1+(P_t-P)T)^{-1}(P_t-P)\mathds 1.
\end{equation}
\end{prop}

Furthermore,  to avoid some tangential convergence making the recurrence criteria  more intricate to expose, we need the following technical hypothesis.

\begin{ass}[sector condition]\label{sector} For any sufficiently small neighbourhood $\mathcal V\subset\mathbb T^d$ of the origin, there exists $K>0$ such that for all  $t\in \mathcal V$,
\begin{equation}\label{sectoreq}
|{\Im}(\lambda(t))|\leq K\,{\Re}(1-\lambda(t)).
\end{equation}
\end{ass}

One can reformulate this condition by saying the family $\{P_t\}_{t\in\mathcal V}$ is uniformly sectorial. Most of the operators encountered are sectorial and one can consult \cite[Chapter 2]{Lunardi1} and  \cite{CGT} for a rigorous definition and elementary properties.

When  the underlying Banach space is  the Hilbert space $\mathbb{L}^2(\pi)$, this assumption can be stated in a more handy way in terms of the associated sectorial forms since  most of the information about the spectrum can be obtained from their numerical range by using the minimax principle. We refer to \cite{Kato} and particularly its Chapters Five  and Six for more details. Introduce the sesquilinear form 
\begin{equation}
\mathcal E_t[f,g]=\sum_{c\in\mathcal S}(f-P_tf)(c)\overline{g(c)}\pi(c).
\end{equation}
Note that for $t=0$, it is nothing but the usual Dirichlet form associated with the driving chain. Then one can consider  the real and imaginary part of the latter form  respectively given by
\begin{equation}
\mathfrak R_t(f,g):=\frac{\mathcal E_t[f,g]+\overline{\mathcal E_t[g,f]}}{2}\quad\mbox{and}\quad \mathfrak I_t(f,g):=\frac{\mathcal E_t[f,g]-\overline{\mathcal E_t[g,f]}}{2i}.
\end{equation}
It turns out  that $\mathfrak R_t$ is a symmetric and positive (definite when $t\neq 0$) sesquilinear form but also that condition  (\ref{sectoreq}) is equivalent to the usual sector condition $|\mathfrak I_t|\leq C\,\mathfrak R_t$.  Typically, this inequality trivially holds when $\pi$ is a reversible probability measure and the conditional jumps satisfy the symmetry relation $\mu_{c,s}(dx)=\mu_{s,c}(-dx)$ for every $c,s\in\mathcal S$. In that case, the imaginary part vanishes and the spectrum is real. Roughly speaking, the sector condition does not allow a too strong drift term.

The following theorem deal with a Fourier-like criterion for MRWs which extends to a series criterion (\ref{rec4}) for more general initial distribution. Note that a distribution $\nu$ induces a continuous linear form  on $\mathbb  L^1(\pi)$ if and only if there exists $c>0$ such that $\nu\leq c\pi$. Such distributions are said to be  dominated by $\pi$.

\begin{theo}[Fourier and series criterion] \label{Zcriterion}
Under Assumptions \ref{assprob}, \ref{assop} and \ref{sector} the MRW is recurrent or transient accordingly as
\begin{equation}\label{fourier}
\lim_{r\uparrow 1}\int_{\mathcal V} {\Re}\left(\frac{1}{1-r\lambda(t)}\right)dt=\infty\quad\mbox{or}\quad 
\lim_{r\uparrow 1}\int_{\mathcal V} {\Re}\left(\frac{1}{1-r\lambda(t)}\right)<\infty,
\end{equation}
for some (or any) neighbourhood $\mathcal V$ of the origin for which $\lambda(t)$ is well-defined. Besides, the integral above is infinite or finite accordingly as 
\begin{equation}\label{series}
\sum_{n=0}^\infty\mathbb P_\nu(Z_n=0)=\infty\quad\mbox{or}\quad \sum_{n=0}^\infty\mathbb P_\nu(Z_n=0)< \infty,
\end{equation}
for some (or any) initial distribution $\nu$ dominated by $\pi$. 
\end{theo}

Observe that the first term in the series expansion (\ref{expans}) is nothing but $\widehat \mu_\pi(t)-1$ where $\mu_\pi$ is defined in (\ref{mixture}). Therefore, the integral criterion (\ref{fourier}) can be interpreted as a perturbation of the classical one obtained for a random walk with $\mu_\pi$ as jump distribution. Also, one can note using the sector condition that this criterion can be rewritten in terms of
\begin{equation}\label{fourier2}
\int_{\mathcal V} \frac{1}{{\Re}(1-\lambda(t))}dt=\infty\quad\mbox{or}\quad 
\int_{\mathcal V} \frac{1}{{\Re}(1-\lambda(t))}dt<\infty.
\end{equation}

Besides, it could be interesting to compare (\ref{fourier}) with the conjecture in \cite[p. 126]{Alsmeyer}. Furthermore, the series criterion (\ref{series}) is obvious for any initial distribution $\nu$ when the internal state space is finite or, to go further, when the internal operator satisfies some Doeblin's condition. It is also possible to suppose the MRW satisfying some scaling limit as in \cite{HervPen} to get the same series criterion.

\subsection{Necessary and sufficient criteria for the quadruple-infinite comb model}

\label{criteriacomb}



 
We first need to extend an oscillation criterion  used in \cite{Rainer2007} for unimodal symmetric distributions.

\begin{theo}[series and Fourier criterion]\label{critère}
  Let $\{H_n\}_{n\geq 0}$ and $\{V_n\}_{n\geq 0}$ be two independent RW on $\mathbb Z$ starting from the origin, the second one being symmetric. Then
\begin{equation}\label{borelcantelli0}
\mathbb P(H_{n+1}=0, V_n V_{n+1}\leq 0,\;\mbox{i.o.})=1\quad \Longleftrightarrow\quad \sum_{n=0}^\infty \mathbb P(H_{n+1}=0)\mathbb P(0\leq V_n\leq V_{n+1}-V_n)=\infty.
\end{equation} 
Furthermore, if the two-dimensional random walk $\{(H_n,V_n)\}_{n\geq 0}$ is transient, this criterion is equivalent to the Fourier integral criterion
\begin{equation}\label{Fourier}
\lim_{r\uparrow 1}\int_{\mathcal V} \int_{\mathcal V}\Re\left(\frac{\Phi_{V}(r,s)}{1-\varphi_{H}(t)\varphi_{V}(s)}\right)ds\,dt=\infty,
\end{equation}
where $\mathcal V\subset\mathbb T^2$ is any sufficiently small neighbourhood of the origin, $\varphi_{H}$ and $\varphi_{V}$ are the characteristic functions of the jumps associated with $H$ and $V$ respectively and $\Phi_{V}$ is the trigonometric series   
\begin{equation}\label{fourierseries}
\Phi_{V}(r,s)=\sum_{n=0}^\infty r^n\mathcal T_{V}(n)\cos(ns).
\end{equation}
Here we denote by $\mathcal T_{V}(n)$ the two-sided tail distribution of the symmetric jumps of $V$. 
\end{theo}

\begin{rem}
If the mass function of symmetric jump distribution is ultimately non-increasing, the $r$-limit  in (\ref{Fourier}) can be dropped and one can set $r=1$ in (\ref{fourierseries}) leading to a simpler  criterion. 
\end{rem}

In order to prove this result when the distributions are symmetric and unimodal in \cite{Rainer2007}, the authors  invoke an appropriate Borel-Cantelli lemma relying crucially on an unimodality assumption.

As far as we are concerned, we only need Lemma \ref{lévysym} below on Lévy concentration functions. It is surprisingly, to our knowledge, not mentioned anywhere. We recall that the Lévy concentration function of a real  random variable $X$ is defined for all $\lambda\geq 0$ by
\begin{equation}
Q(X,\lambda)=\sup_{x\in\mathbb R}\mathbb P(x\leq X\leq x+\lambda).
\end{equation}
The following fundamental Lemma means -- roughly speaking --  that the supremum is reached near the origin for the symmetric distributions.

\begin{lem}[Lévy concentration function of symmetric random walks]\label{lévysym}
Let $\{M_n\}_{n\geq 0}$ be a symmetric random walk. Then there exists two positive universal constants $L$ and $C$  such that for all $p>0$ for which the characteristic function of the jumps is non-negative on $[-p,p]$ and all $n\geq 1$ and $\lambda\geq L/p$,
\begin{equation}\label{maj1}
\mathbb P(0\leq M_n\leq \lambda)\leq Q(M_n,\lambda)\leq C\,\mathbb P(0\leq M_n\leq \lambda).
\end{equation}
\end{lem}

\begin{rem}
One can obtain similar bounds replacing the condition $0\leq M_n\leq \lambda$  in (\ref{maj1}) by the symmetric one  $|M_n|\leq \lambda/2$. Besides, one can note that $Q(M_n,\lambda)=\mathbb P( |M_n|\leq  \lambda/2)$ when the jump distribution is unimodal and non-atomic. 
\end{rem}

At this level, we may apply Theorem \ref{critère} to provide a necessary and sufficient criteria for a wide class of PRWs, built from a quadruple infinite comb as in Section \ref{quadruple}, under the following assumptions.

\begin{ass}[generalized DRRWs]\label{modelrestrein} Let $\{S_n\}_{n\geq 0}$ be a quadruple infinite comb PRW starting from the origin, the initial time  being a vertical-to-horizontal change of direction as in Figure \ref{marche}, such that
\begin{enumerate}
\item[\rm (H1)] the persistence times when the walker moves horizontally or vertically are independent of each other and {\it i.i.d.\@}. Their distributions are respectively denoted by $\nu_{\mathtt h}(dt)$ and  $\nu_{\mathtt v}(dt)$;
\item[\rm (H2)] the probabilities to change from the current direction into an orthogonal one only depend on the final direction (among east, north, west or south) and are constant with respect to the absolute directions (horizontal or vertical).
Those are denoted by  
\begin{equation}
p_{\mathtt e}=p_{\mathtt w}=\frac{1-p_{\mathtt v}}{2}\quad\mbox{and}\quad p_{\mathtt n}=p_{\mathtt s}=\frac{1-p_{\mathtt h}}{2},
\end{equation}
in such way that $p_{\mathtt h}$ and $p_{\mathtt v}$ stand respectively for the probabilities to stay in the current horizontal and vertical direction at each breaking time.  
\end{enumerate}
\end{ass}

This framework includes two types of PRWs which are of particular interest when the waiting time distributions are equal, that is $\nu_{\mathtt h}(dt)=\nu_{\mathtt v}(dt)$:
\begin{itemize}
\item Original DRRWs if $p_{\mathtt h}=p_{\mathtt v}=1/3$.
\item DRRWs without U-turns (non-backtracking DRRWs) if $p_{\mathtt h}=p_{\mathtt v}=0$.
\end{itemize}

Non-backtracking DRRWs are natural generalizations of the symmetric one-dimensional PRWs investigated in \cite{PRWI} and was the original motivation of this work. 

Let us introduce a symmetric Rademacher random variable $\varepsilon$, two geometric random variables  $G_\mathtt h$ and $G_\mathtt v$ with parameters $1-p_{\mathtt h}$ and $1-p_{\mathtt v}$, and two sequences of {\it i.i.d.\@} random variables $\{\tau_{k}^{\mathtt h}\}_{k\geq 1}$ and $\{\tau_{k}^{\mathtt v}\}_{k\geq 1}$ distributed as $\nu_{\mathtt h}(dt)$ and $\nu_{\mathtt v}(dt)$. We assume that all of these are independent of each others. Then we can consider two independent random walks $\{H_n\}_{n\geq 0}$ and $\{V_n\}_{n\geq 0}$ whose respective jumps are distributed as 
\begin{equation}\label{geometric}
\varepsilon\sum_{k=1}^{G_{\mathtt h}}(-1)^{k-1}\tau_k^{\mathtt h}\quad\mbox{and}\quad \varepsilon\sum_{k=1}^{G_{\mathtt v}}(-1)^{k-1}\tau_k^{\mathtt v},
\end{equation}
and state a necessary and sufficient criterion for the recurrence of these specific PRWs.

\begin{cor}[necessary and sufficient criteria for generalized DRRWs]\label{coro} 
Under Assumption \ref{modelrestrein} the origin is recurrent for $\{S_n\}_{n\geq 0}$ if and only if
\begin{equation}\label{borelcantelli}
\sum_{n=0}^\infty \mathbb P(H_{n+1}=0)\mathbb P(0\leq V_{n}\leq V_{n+1}-V_n)=\infty\;\mbox{ or }\;\sum_{n=0}^\infty \mathbb P(V_{n+1}=0)\mathbb P(0\leq H_n\leq H_{n+1}-H_n)=\infty.
\end{equation}
\end{cor}

Thereafter, we answer by the negative to the conjecture in \cite{Mauldin1996} and moreover produce a constructive method to build recurrent PRWs with transient MRW skeletons.

\begin{theo}[definitive invalidation of the conjecture and more]\label{conjecture}
There exists waiting time distributions on the positive integers $\{1,\cdots\}$  such that the associated DRRWs  and non-backtracking DRRWs in $\mathbb Z^2$ are recurrent whereas their MRW skeletons are transient.
\end{theo}

We can deduce from the Fourier criterion (\ref{Fourier}) that  such distributions are necessarily non-integrable and we provide a generic and inductive construction. Note that in the case of non integrable persistent times, there is no invariant probability measure for the associated VLMC. Inspired by \cite{Shepp}, one could ask for a proof relying on Fourier analysis. Such an approach has seemed to us tedious. That is why our preferences go to a more concrete probabilistic proof in the spirit of \cite{Grey,Rainer2007}.

\section{Proofs of Section \ref{secFourier}}

\label{prooffourier}

\setcounter{equation}{0}

We begin with Proposition \ref{eigenvalue} which lay the groundwork for Theorem \ref{Zcriterion} and Proposition \ref{expansprop}.

\begin{proof}[Proof of Proposition \ref{eigenvalue}]
First, we get from Assumptions ($\hyperlink{P1}{\rm P1}$) and ($\hyperlink{O1}{\rm O1}$) that ${\rm ker}(P-I)=\mathbb C.\mathds 1$ in $\mathcal B$. Together with the spectral gap condition  ($\hyperlink{O2a}{\rm O2a}$) and since any isolated element of the spectrum is an eigenvalue, the spectral radius of $P$ is necessarily  equal to $1$. Besides, the eigenvalue $1$ is necessarily of algebraic multiplicity one. Otherwise, the operator $P-I$ would induce a linear and surjective map from ${\rm ker}(P-I)^2\supsetneq \mathbb C.\mathds 1$ to $\mathbb C.\mathds 1$ and thus there would exist $f\in\mathcal B$ such that $Pf=f+\mathds 1$, in contradiction with the Perron-Frobenius Theorem. Furthermore, an other application of this theorem shows  us that if $\lambda$ is an eigenvalue of $P$ on the unit circle then $\lambda=1$. It remains to extend continuously those results on a neighbourhood of the origin by the mean of the perturbation theory.

The existence of a neighbourhood $\mathcal V$ of the origin such that the first and the second points of Proposition \ref{eigenvalue} are satisfied follows directly from \cite[Theorem 3.16., p.212]{Kato} together with the latter considerations and the continuity  hypothesis (\hyperlink{O2b}{02b}). To prove the third point, we also use the perturbation theory but we need to take care about the (possibly) infinite dimensional situation when we use the Cauchy holomorphic functional calculus. 

In fact, applying more precisely \cite[Theorem 3.16., p.212]{Kato}, one deduce  there exist $\delta>0$ and  a positively-oriented curve $\Gamma$ enclosing $1$ such that for any bounded linear perturbation $H$ smaller that $\delta$ there exists a unique element $\lambda(H)$ of maximal modulus in the spectrum of $P+H$. The latter is again an eigenvalue of algebraic multiplicity one but also the unique element of the spectrum inside $\Gamma$. Besides, since the resolvent $R_H(\xi):=(P+H-\xi)^{-1}$ is holomorphic  outside the (compact) spectrum of $P+H$ we can consider the following so called Dunford integral
\begin{equation}\label{projector}
Q_H:=-\frac{1}{2\pi i}\int_{\Gamma} R_H(\xi)d\xi,
\end{equation} 
which do not depend on such $\Gamma$. It turns out that $H\longmapsto Q_H$ is continuous in a neighbourhood of the origin. Writing the Laurent series expansion of $R_H(\xi)$ around $\lambda (H)$ it is classical that $Q_H$ is the continuous projector on the generalized eigenspace associated with $\lambda (H)$. Since  the latter  is  of multiplicity one,  this space is one-dimensional and thus  
\begin{equation}\label{lambda}
\lambda (H)-1={\rm Tr}(Q_H (P+H))-1={\rm Tr}((P+H-1) Q_H).
\end{equation}
Here we denote by ${\rm Tr}$ the linear trace defined on the finite-rank operator ideal. 

\begin{lem}\label{rankone}\label{tracecont}
The trace operator is continuous on the space of rank-one bounded linear operators  endowed with the induced subordinated norm distance.
\end{lem}

\begin{proof}
Let $(T_n)$ be a sequence of continuous rank-one operators converging  to $T$ for the subordinated norm. There exist continuous linear forms $(\varphi_n)$ and $\varphi$  and vectors $(v_n)$ and $v$ such that 
\begin{equation*}
{\rm Tr}(T_n)=\varphi_n(v_n)\quad\mbox{and}\quad {\rm Tr}(T)=\varphi(v),
\end{equation*}
where $T_n$ and $T$ are respectively represented as $\varphi_n\otimes v_n$ and $\varphi\otimes v$. Besides, we can assume that $(v_n)$ and $v$ are of norm $1$ and  then  necessarily  $v_n\longrightarrow v$ and $\varphi_n\longrightarrow \varphi$ for the subordinated norm. In particular, we deduce the convergence ${\rm Tr}(T_n)\longrightarrow {\rm Tr}(T)$.
\end{proof}

Finally, the continuity of the perturbed eigenprojector, the  representation (\ref{lambda}) and  Lemma \ref{tracecont} above imply that $H\longmapsto\lambda (H)$ is continuous  on a neighbourhood of the origin.  By using (\hyperlink{O2b}{O2b}) we get  the continuity of $\lambda(t)$ on a neihgbourhood of the origin  since we can write $\lambda(t)=\lambda (P_t-P)$.

It remains to prove the last and fourth point. Let us denote by $\rho(T)$ the spectral radius of a bounded linear operator $T$ on a Banach space. Since $T\longmapsto \rho(T)$ is upper semi-continuous, so is $t \rightarrow \rho(P_t)$ by the continuity assumption. It follows that $t \rightarrow \rho(P_t)$ reaches its maximum $M$ on any compact set $K\subset\mathbb T^d$ at some point $t^\ast\in K$. Let $\lambda$ be a peripheral spectral value of $P_{t^\ast}$ so that $|\lambda|=M$. It comes from Assumption (\hyperlink{O2c}{O2c}) that $P_{t^\ast}h=\lambda h$ for some eigenvector $h\in\mathcal B$. Since the modulus of a characteristic function is lower than one, the triangle inequality gives for every $c\in\mathcal S$,
\begin{equation}\label{maj2}
M |h(c)|\leq 
\sum_{s\in\mathcal S} |\widehat{\mu}_{c,s}(t^\ast)||h(s)| P(c,s)
\leq P|h|(c).
\end{equation}
Note that $M \leq 1$ and that $P^n|h|$ converge pointwise towards $\pi(|h|) \mathds 1$ by ergodicity. Suppose now $M=1$ and $t^\ast\neq 0$. By using (\ref{maj2}) and the latter two remarks we get $|h| \leq \pi(|h|) \mathds 1$ and thus $|h|\in{\rm span}(\mathds 1)$. This implies that $|\widehat \mu_{c,s}(t)|=1$ as soon as $P(c,s)\neq 0$. However, this property is equivalent for the MRW to be periodic, which is excluded. As a consequence, for any neighbourhood $\mathcal V\subset\mathbb T^d$ of the origin,
\begin{equation}\label{rayonspectral}
\sup_{t\in \mathbb T^d\setminus \mathcal V} \rho(P_t)<1.
\end{equation}
Since $|\lambda(t)|=\rho(P_t)$ on $\mathcal V$, the proof of the proposition is completed.
\end{proof}

\begin{proof}[Proof of Theorem \ref{Zcriterion}]
First, the Markov additive property implies for every $n\geq 1$ and $\pi$-integrable or non-negative function $f$ on $\mathcal S$ the identity    
\begin{equation}\label{puissance}
P_t^n f(c)=\mathbb E_c[e^{it.Z_n}f(C_n)].
\end{equation}
We show that 
\begin{equation}\label{limitr}
\sum_{n \geq 0} \mathbb{P}_{\nu}(Z_n=0) = \lim_{r \uparrow 1} \sum_{n \geq 0} r^n \mathbb{P}_\nu(Z_n=0) = \lim_{r \uparrow 1}\int_{\mathbb{T}^d} \Re(\nu(1-rP_t)^{-1} \mathds 1)\,dt.
\end{equation} 
To this end, write the resolvent operator $(1-rP_t)^{-1}$ on $\mathcal B$ as the classical series expansion of the   bounded operators $r^n P_t^n$ when $0<r<1$. Also, remark  that  $t\longmapsto \nu (r^n P_t^n)\mathds 1$ is continuous by assumptions  and bounded by $r^n$ from (\ref{puissance}). Then applying the dominated convergence theorem, we deduce (\ref{limitr}).

To go further, let $0<\epsilon<1$ be given by (\ref{rayonspectral}) so that $\rho(P_t)\leq 1-\epsilon$ for every $t\notin\mathcal V$. Again, it follows from the series expansion of the resolvent that there exists $C>0$ such that $\|(1-rP_t)^{-1}\|_{\mathcal B}\leq C\epsilon^{-1}$ for every $0<r<1$ and $t\notin \mathcal V$. Since $\nu$ is assumed to be a  continuous linear form on $\mathbb{L}^1(\pi)$, it is also continuous on $\mathcal B$ from (\hyperlink{O1}{O1}) -- say of norm $N_1$. Denoting by $N_2$ the $\mathcal B$-norm of $\mathds 1$ we get
\begin{equation}\label{maj-exterieur}
\left|\lim_{r\uparrow 1}\int_{\mathbb{T}^d\setminus\mathcal V } \Re(\nu(1-rP_t)^{-1} \mathds 1) dt\right|\leq C\epsilon^{-1}N_1N_2<\infty.
\end{equation} 
Therefore, the finiteness or not of the $r$-limit on the right-hand side of (\ref{limitr}) is completely given by the same $r$-limit but integrating on any (or some) neighbourhood of the origin.

Let us write $P_t=\lambda(t)Q_t+E_t$ where $Q_t:=Q_{P_t-P}$ is the one-dimensional projector on the eigenspace associated with $\lambda(t)$ defined by (\ref{projector}). Note that  $E_t$ can be seen as the restriction of $P_t$ to the stable subspace ${\ker}(Q_t)={\Im}(1-Q_t)$ and $Q_t$ the restriction of $P_t$ to the one-dimensional supplementary subspace ${\ker}(1-Q_t)={\Im}(Q_t)$. Besides, another use of  \cite[Theorem 6.17., p. 178]{Kato} with the  spectral gap condition (\hyperlink{O2b}{O2b}) gives $0<\epsilon<1$ such that $\rho(E_t)\leq 1-\epsilon$ for every $t\in\mathcal V$. In addition, the operators $Q_t$ and $E_t$ commute so that $P_t^n=\lambda(t)^n Q_t+E_t^n$ for every $n\geq 1$. It comes 
\begin{equation*}
\Re(\nu(1-rP_t)^{-1} \mathds 1)=\Re\left(\frac{\nu Q_t \mathds 1}{1-r\lambda(t)}\right)+\Re(\nu(1-rE_t)^{-1}\mathds 1).
\end{equation*}
As for (\ref{maj-exterieur}), similar arguments imply that the second term in the right-hand side of the latter equality is bounded by some positive constant, uniformly with respect to $0<r<1$ and $t\in\mathcal V$, in such way that the finiteness or not of the $r$-limit depend only on the first term. Moreover, this latter integrand  can be rewritten up to the multiplicative term $|1-r\lambda(t)|^{-2}$ as 
\begin{equation*}
{\Re}(\nu Q_t\mathds 1){\Re}(1-r{\lambda(t)})-{\Im}(\nu Q_t\mathds 1) {\Im}(r\lambda(t)).
\end{equation*}
Note that  $|{\Im}(r\lambda(t))|\leq K\,{\Re}(1-r{\lambda(t)}) $ for every $0<r<1$ and $t\in\mathcal V$ by  the sector Assumption \ref{sector}. Also, remark  that $\nu Q_t\mathds 1$ converges toward $1$ as $t$ goes to $0$. Then, we deduce easily that  
\begin{equation*}
\sum_{n \geq 0} \mathbb{P}_{\nu}(Z_n=0)=\infty
\quad\Longleftrightarrow\quad
\lim_{r\uparrow 1}\int_{\mathcal V}{\Re}\left(\frac{1}{1-r\lambda(t)}\right) dt=\infty,
\end{equation*}  
for some (or any) neighbourhood of the origin for which $\lambda(t)$ is well-defined. It is worth noting that that $\nu$  disappears of the integral condition. Therefore, the necessary and sufficient recurrence criterion follows from Proposition \ref{dicho2}. This ends the proof of the Theorem.
\end{proof}

\begin{proof}[Proof of Proposition \ref{expansprop}]
We shall prove the series expansion (\ref{expans}). To this end, we continue the work initiated for (\ref{lambda}). Since $(P+H-1)R_H(\xi)=1+(\xi-1)R_H(\xi)$ one has 
\begin{equation*}
\lambda (H)-1={\rm Tr}\left(\frac{1}{2\pi i}\int_{\Gamma}(1-\xi)R_H(\xi)d\xi\right).
\end{equation*}
Besides, one has for sufficiently small perturbation,
\begin{equation*}
R_H(\xi)=R(\xi)\sum_{p\geq 0}(-HR(\xi))^p.
\end{equation*}
Here we denote $R:=R_0$ the resolvent of $P$. Then remark that one can exchange the integral and the latter series so that the problem reduces to the study of the trace associated with the absolute convergent series of bounded operators
\begin{equation}\label{integrals}
\sum_{p\geq 0}\frac{1}{2\pi i}\int_\Gamma (1-\xi)R(\xi)(-HR(\xi))^p~d\xi.
\end{equation}  
Consider the Laurent series expansion of the resolvent given by
\begin{equation}\label{Laurent}
R(\xi)=-\frac{Q}{\xi-1}+\sum_{n\geq 0}(\xi-1)^n T^{n+1}.
\end{equation}
Recall that $Q$ is the eigenprojector on ${\rm span}(\mathds 1)$ and $T$ is the operator defined in (\ref{reducedresolvent}). For more details, we refer to \cite[Chap. I.5.3.]{Kato}. Thereafter, to evaluate each integrals in (\ref{integrals}) we need to identify the principal singularity of the integrand. Using (\ref{Laurent}), it is nothing but 
\begin{equation}\label{singularity}
(-1)^{p-1}Q\left[HQ(HT)^{p-1}+(HT)HQ(HT)^{p-2}+\cdots+(HT)^{p-1}HQ\right],
\end{equation}
for any $p\geq 1$. Furthermore, since these operators are one-dimensional and the series is absolutely convergent, we get from Lemma \ref{rankone} that the trace operator and the series commute. In addition, we deduce from the formal rule ${\rm Tr}(AB)={\rm Tr}(BA)$ when $A$ or $B$ is of rank one and from the identity $TQ=0$ that the only operator in (\ref{singularity}) having a possibly non-zero trace is the last one. It is then not difficult to write its contribution as
\begin{equation*}
(-1)^{p-1}{\rm Tr}(Q(HT)^{p-1}HQ)=(-1)^{p-1}\pi(HT)^{p-1}H\mathds 1.  
\end{equation*}
Summarizing,  
\begin{equation*}
\lambda (H)-1=\sum_{n\geq 0}(-1)^{n}\pi(HT)^{n}H\mathds 1=\pi(1+HT)^{-1}H\mathds 1,
\end{equation*}
for sufficiently small perturbations. This leads to (\ref{expans}) and thus  ends the proof.
\end{proof}

\section{Proofs of Section \ref{criteriacomb}}

\label{proofcomb}

\setcounter{equation}{0}

We begin with the proof of Lemma \ref{lévysym}. Thereafter, we will be able to prove Theorem \ref{critère} applying a generalized Borel-Cantelli argument and then obtain Corollary \ref{coro}. The proof of Theorem \ref{conjecture} requires these three results and is given at the end of this section.

\begin{proof}[Proof of Lemma \ref{lévysym}] We follow and make more precise the results of \cite{Esseen1,Esseen2} connecting the behaviour of the Lévy concentration function with the integral near the origin of the characteristic function. Let $\varphi(t)$ be the characteristic function of the jumps associated with $\{M_n\}_{n\geq 0}$. Since the latter is symmetric, one can find $p>0$ such that $\varphi(t)\geq 0$ on $[-p,p]$. Let us introduce now  
\begin{equation*}
h(t):=(1-|t|)^+\quad\mbox{and}\quad H(x):=\int e^{ixt}h(t)dt=\left(\frac{\sin(x/2)}{x/2}\right)^2. 
\end{equation*}
A direct consequence of the Fourier-duality implies, for any $\lambda>0$ and $n\in\{1,2,\cdots,\}$, the following crucial identity  
\begin{equation*}\label{crucial}
\frac{\lambda}{2\pi}\int_{-{2\pi}/{\lambda}}^{{2\pi}/{\lambda}} \varphi(t)^n h(\lambda t/2\pi)e^{-it\xi} dt=\int H(2\pi(x-\xi)/\lambda)\mathbb P_{M_n}(dx).
\end{equation*}
Setting $\xi=0$ in the equality above, it follows, for any $\lambda>2\pi/p$ and $n,N\in\{1,2,\cdots\}$, that
\begin{equation}\label{crucial2}
\frac{\lambda}{4\pi}\int_{-\pi/\lambda}^{\pi/\lambda} \varphi(t)^n dt \leq 2\mathbb P\left(0\leq M_n\leq N\lambda \right)+\left(2\sum_{k\geq N}\frac{1}{\pi^2k^2}\right) Q\left(M_n,\lambda\right).
\end{equation}
To conclude, we need the following well known result. Because its proof does not appear clearly in the literature, a brief proof of this fact is given below. 


\begin{lem}\label{encadrement2}
There exist universal constants $0<m\leq M$ such that for any $\lambda>2\pi/p$ and $n\geq 1$,
\begin{equation}\label{encadrement}
m\lambda \int_{-{\pi}/{\lambda}}^{{\pi}/{\lambda}}\varphi(t)^n dt\leq Q(M_n,\lambda)\leq M\lambda \int_{-{\pi}/{\lambda}}^{{\pi}/{\lambda}}\varphi(t)^n dt.
\end{equation}
\end{lem}

\begin{proof}[Proof of Lemma \ref{encadrement2}]
First, looking at \cite[p. 211]{Esseen1}, the upper bound is immediate with an absolute value around the characteristic function.  Since the latter is positive on $[-p,p]$ we get the right-hand side of (\ref{encadrement}). To get a similar lower bound, one can adapt \cite[p. 292]{Esseen2} since the proof remains valid without symmetrization. Again, we only need to care about the set of positivity of $\varphi(t)$.
\end{proof}


Thus let us choose $N$ such that the series in (\ref{crucial2}) is lower than $1/(8\pi M)$.
Then the right-hand side of (\ref{encadrement}), the inequality (\ref{crucial2}) and classical results, stated for instance at the beginning of \cite{Esseen2}, imply the inequality of Lemma \ref{lévysym} with $\lambda_0:=2\pi N/p$ and $C:=16\pi M(N+1)$.
\end{proof}


\begin{proof}[Proof of Theorem \ref{critère}]
First, it comes from the symmetry of $\{V_n\}_{n\geq 0}$ and its increments, denoted by $\{Z_n\}_{n\geq 0}$,  that it is only needed to focus on the events $E_{n}:=\{H_{n+1}=0, 0\leq V_{n}\leq Z_{n+1}\}$. Besides, denoting by $\mathcal T_V$ the right tail distribution of the  jumps of $V$,  conditioning  successively with respect to the  filtrations generated by $\{H_{k} : 1 \leq k\leq n+1\}$ and $\{(V_k,Z_k) : 1\leq k\leq n\}$, and finally applying the usual conditional Borel-Cantelli lemma, we get that
\begin{equation*}
\{E_{n},\;\mbox{i.o.}\}=\left\{\sum_{n\geq 0}\mathds 1_{\{H_{n+1}=0\}}\mathcal T_V(V_{n})\mathds 1_{\{V_{n}\geq 0\}}=\infty\right\}\quad\mbox{a.s.}.
\end{equation*}
Consequently, one can replace the $Z_{n}$ in the definition of the $E_n$ by an identically distributed sequence $\{Z_n^\perp\}_{n\geq 1}$  independent of $\{V_n\}_{n\geq 1}$ since the resulting events -- say $E_n^{\perp}$ -- lead exactly to the same criterion. Furthermore, the limit superior of these events belong to the exchangeable $\sigma$-algebra associated with an {\it i.i.d.\@} sequence of random variables in such way that the Hewitt-Savage zero-one law applies and we only need to prove that $E_n^\perp$ occur infinitely often with a positive probability.

To this end, we observe using conditional arguments and Lemma \ref{lévysym} that for any $n>k\geq 1$,
\begin{equation*}
\mathbb P(E_n^\perp\cap E_k^\perp)\leq \mathbb P(E_{k}^\perp)\mathbb E[\mathds 1_{\{H_{n-k}=0\}}\mathbb P_{V_{k}}(0\leq V_{n-k}\leq Z_{n+1}^\perp)]\leq C\,\mathbb P(E_{k}^\perp)\mathbb P(E_{n-k-1}^\perp).
\end{equation*}
Thereafter, we can conclude with a classical step -- see \cite[p. 726]{Raugi} for instance. In fact, the inequality above implies the sequence $\sum_{k=1}^n \mathds 1_{E_k^\perp}/\sum_{k=1}^n\mathbb P(E_k^\perp)$ is bounded in $\mathbb L^2$ and thus equi-integrable. Then we can apply the generalized Fatou-Lemma so that
\begin{equation*}
\mathbb E\left[\limsup_{n\to\infty}\frac{\sum_{k=1}^n \mathds 1_{E_k^\perp}}{\sum_{k=1}^n \mathbb P(E_k^\perp)}\right]\geq 1.
\end{equation*}
Therefore, if the sequence of partial sums at the denominator is unbounded, then the series at the numerator is divergent with a positive probability. This ends the proof of the series criterion  since the reciprocal implication is a straightforward consequence of the standard Borel-Cantelli Lemma.  

Regarding the Fourier-like criterion (\ref{Fourier}), remark that 
\begin{equation}
 \mathbb P(H_{n+1}=0)=\int_{\mathbb T^d} \Re(\varphi_H^{n+1}(t))\, dt\quad\mbox{and}\quad \mathbb P(0\leq V_n\leq Z_{n+1})=\sum_{k=0}^\infty\tail_V(k)\int_{\mathbb T^d} \cos(ks)\varphi_V^{n}(s)\,ds.
\end{equation}
Then, multiplying by the geometric terms $u^{n}$ and $r^k$ respectively and using standard inversion theorems, it follows the series in the criterion is infinite if and only if  
\begin{equation}
\lim_{r,u\uparrow 1}\int_{\mathbb T^d} \int_{\mathbb T^d}\Re\left(\frac{\varphi_H(t)}{1-u\varphi_{H}(t)\varphi_{V}(s)}\right)\Phi_{V}(r,s)\,ds\,dt=\infty.
\end{equation}
When $(H,V)$ is transient, it follows from the Ornstein-Chung-Fuchs  criterion  that  the $u$-limit can be remove since $1/(1-u\varphi_{H}(t)\varphi_{V}(s))$ is uniformly integrable on $\mathbb T^d\times \mathbb T^d$ for  $0<u<1$. It is then  not difficult to drop  $\varphi_H(t)$ and integrate around the origin to get the integral criterion.
\end{proof}


\begin{proof}[Proof of Corollary \ref{coro}]
First remark that the original generalized DRRW pass through the origin during either a horizontal or vertical move. Besides, it turns out that  $\{(H_n,V_n)\}_{n\geq 0}$ and $\{(H_{n+1},V_n)\}_{n\geq 0}$ are respectively the skeleton random walks associated with the horizontal-to-vertical and the vertical-to-horizontal changes of direction. Due to Theorem \ref{critère}, the recurrence of the origin follows from the divergence of one of the series in (\ref{borelcantelli}), each of them corresponding to a walker passing through the origin infinitely often during vertical or horizontal moves respectively.

Thus, it remains to prove that the convergence of both series in (\ref{borelcantelli}) leads to the transience of the origin. We only consider the first series since the other one can be treated analogously. With the settings in (\ref{geometric}), we observe that it suffices to show 
\begin{equation}\label{borelcantelli2}
\sum_{n=0}^\infty \mathbb P(H_{n+1}=0)\mathbb P\left(0\leq V_{n}\leq \max_{1\leq l \leq G_{\mathtt v}}\varepsilon A_l \right)<\infty,\quad\mbox{where}\quad A_l:=\sum_{k=1}^{l}(-1)^{k-1}\tau_{k}^{\mathtt v}. 
\end{equation}

Here, all the involved random variables are independent. In particular, applying the standard Borel-Cantelli Lemma,  we deduce from (\ref{borelcantelli2}) that the origin is not recurrent -- for a walker passing through the origin during vertical moves. Since a similar argument holds for horizontal movements, it suffices to check that the divergence of the first series in (\ref{borelcantelli})
 implies  (\ref{borelcantelli2}) to terminate the proof. Finally, this is obvious when $p_{\mathtt v}=0$, otherwise it is a consequence of the following lemma.
\begin{lem} The following estimate holds for all $n \geq 0$
\begin{equation*}\label{russe}
\mathbb P\left(0\leq V_{n}\leq \max_{1\leq l \leq G_{\mathtt v}}\varepsilon A_l  \right)\leq 2(1+p_{\mathtt v})\,\mathbb P\left(0\leq V_{n}\leq V_{n+1}-V_n \right).
\end{equation*}
\end{lem}

\begin{proof}
Note that $A_{2i+1}\geq A_{2i+2}$ for every $i\geq 0$.
Hence, the local maxima of $A_l$ are reached for odd indices. Since $A_{2i+1}$ can be rewritten as $A_{2i+1}=\tau_{1}^{\mathtt v}+\sum_{k=1}^i (\tau_{2k+1}^{\mathtt v}-\tau_{2k}^{\mathtt v})$, a direct application of \cite[Chapter 3, Theorem 10, p.50]{Petrov} for the symmetric increments $\tau_{2k+1}^{\mathtt v}-\tau_{2k}^{\mathtt v}$ leads to 
\begin{equation*}
\mathbb P\left(\max_{1\leq l\leq 2i+2}A_l\geq x\right)=\mathbb P\left(\max_{1\leq l\leq 2i+1}A_l\geq x\right)\leq 2 \mathbb P\left(A_{2i+1}\geq x\right),
\end{equation*} 
for every $x\in\mathbb R$.
The result then follows by conditioning with respect to the event  $\{\varepsilon=1\}$ and the random variables $V_n$ and $G_{\mathtt v}$.
\end{proof}
This ends the proof of Corollary \ref{coro}.
\end{proof}


\begin{proof}[Proof of Theorem \ref{conjecture}] We shall construct inductively appropriate jump distributions as in the concise and elegant paper  of \cite{Grey}. In this article is given a probabilistic proof of a result of \cite{Shepp}. This result states that a recurrent symmetric random walk on the line may have jump with arbitrary large tails. We also follow the clever path borrowed by the authors in \cite{Rainer2007}. We stress that we do not require unimodality assumptions making the construction more general but also handier and easier to state.\\

\noindent
{\bf Step 1.} In order to introduce the suitable distributions, we  define inductively a sequence $\{ (l_k,y_k,p_{k+1}) \}_{k\geq 1}$ where $l_k$ and $y_k$ are non-negative integers satisfying  $y_{k+1}\geq y_{k}+l_{k}$. The couple $l_k$ and $y_k$, $k \geq 1$, represent some spatial parameters explained in step 2.\@ whereas $p_{k+1}$ stands for some probability. 

The following quantities will be fixed throughout all the procedure: $(v_k)_{k\geq 2}$ and $(u_k)_{k\geq 2}$ are two sequences of positive numbers such that some fixed $\delta>0$ and $c>0$ one has for every $k\geq 2$, 
\begin{equation}\label{divergence}
v_k\underset{k\to\infty}{=}{\rm o}(u_k)\quad \mbox{and}\quad \frac{1}{v_k}<\frac{c}{k^{2+\delta}};
\end{equation}
$\alpha$ and $\beta$ are positive constant.  Then, proceed as follows: choose $y_1,l_1\geq 1$ and, given some fixed $r\in(0,1)$, choose $0<p_2<1-r$. Knowing the first $k-1$ terms of the sequence $\{(l_k,y_k,p_{k+1})\}_{k \geq 1}$,  we may choose $(l_k,y_k,p_{k+1})$ respecting the following constraints:

\begin{itemize}
\item first  choose $(l_k,y_k)$ such that $y_k\geq y_{k-1}+l_{k-1}$ and $l_k\geq 2$ for all $k$ sufficiently large with  
\begin{flalign}
\quad \mbox{\bf --}\quad  & \frac{1}{l_k^2 p_k^2}\ln\left(\frac{1}{rp_k}\right)\leq \frac{1}{v_k}\label{CondConst1}; & \\[5pt]
\mbox{\bf--}\quad & \sum_{i=1}^{k-1}p_i(y_i+l_i)^2\leq \alpha p_k(y_k+l_k)^2; \label{CondConst2}& \\[5pt]
\mbox{\bf--}\quad & \frac{p_k^2 l_k^2 y_k}{(y_k+l_k)^2}\geq u_k;\label{CondConst3}\\[5pt]
\mbox{\bf--}\quad & \frac{y_k^2}{p_k(y_k+l_k)^2}\leq \beta;\label{CondConst4}
\end{flalign}
\item in  second step, choose $p_{k+1}$ such that
\begin{flalign}
\quad \mbox{\bf--}\quad  & \frac{1}{v_k} \leq \frac{1}{l_k^2 p_k^2}\ln\left(\frac{1}{p_{k+1}}\right)\leq \frac{c}{k^{2+\delta}}. \label{CondConst5}&
\end{flalign} 
\end{itemize}

Note that such a triplet $(l_k,y_k,p_{k+1})$ does exist. Actually, it is even possible to choose  $y_k=l_1+\cdots+l_{k-1}$.  Together with \eqref{CondConst1}, the condition \eqref{CondConst5} implies $p_{k+1}\leq r p_k$. It then follows that  $p_2+p_3+\cdots$ is lower than $p_2/(1-r)<1$ so that one can choose $0\leq p_0\leq 1-p_2/(1-r)$ arbitrary  and find $0\leq p_1\leq 1$ such that
\begin{equation*}
q_k:=1-(p_0+\cdots+p_k)\xrightarrow[k\to\infty]{}0.
\end{equation*}

\noindent
{\bf Step 2.} We shall associate with the sequence $\{(l_k,y_k,p_{k+1})\}_{k \geq 1}$ a sequence of distributions defining coupled defective random walks, denoted in the sequel by $\{ W_n^k \}_{n \geq 0}$, $k \geq 1$, such that the limiting random walk $\{H_n\}_{n\geq 0}$ -- see below for the precise meaning -- obtained by these approximations  satisfies 
\begin{equation}\label{majfinal}
{\sum_{n=1}^\infty \mathbb P(H_n=0)^2}<\infty\quad\mbox{and}\quad \sum_{n=1}^\infty\mathbb P(H_n=0)\mathbb P(0\leq H_n\leq H_{n+1}-H_n)=\infty.
\end{equation}

\begin{figure}[!h]
\begin{center}
\includegraphics[width=80mm]{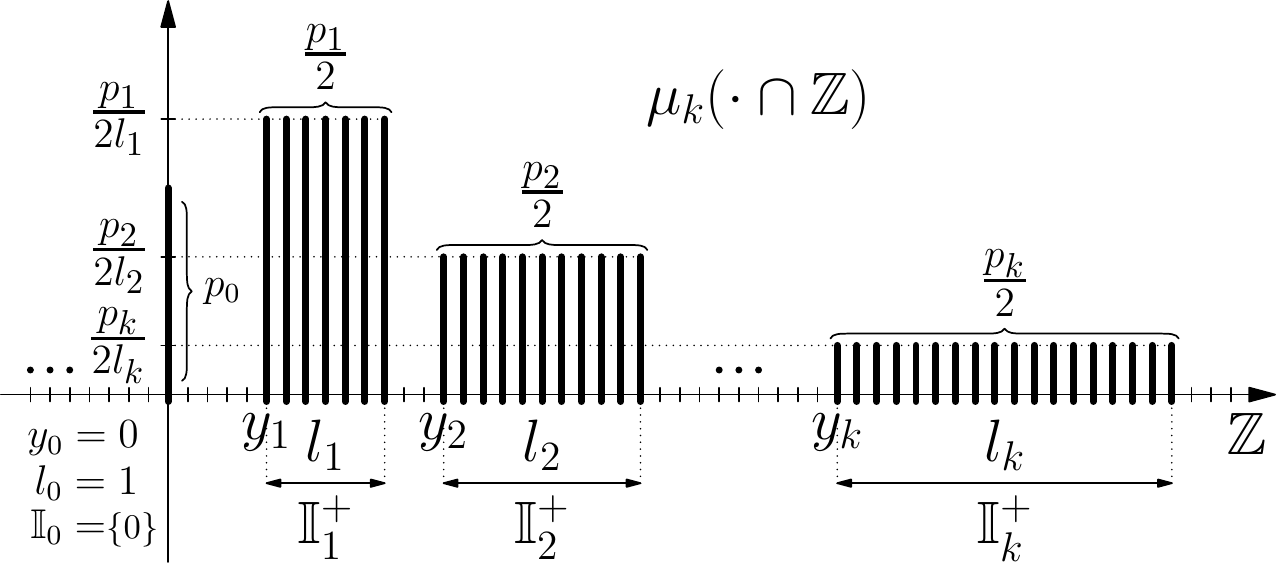}
\end{center}
\caption{\label{distrib} The $k${\rm th} symmetric distribution}
\end{figure} 

To this end, adjoin a cemetery $\Delta\notin\mathbb Z$. We set $y_0=0$, $l_0=1$ and $\mathbb I_0=\{0\}$ and for every $k\geq 1$,
\begin{equation*}
\mathbb I_k:=\mathbb I_k^+\sqcup \mathbb I_k^-\quad\mbox{with}\quad \mathbb I_k^+:=[y_k,y_k+l_k)\cap \mathbb Z\quad\mbox{and}\quad \mathbb I^-_k:=-\mathbb I_k^+.
\end{equation*}
Then we consider for every $k\geq 0$ the distribution $\mu_k$ on $\mathbb I_0\sqcup \cdots \sqcup \mathbb I_k\sqcup \{\Delta\}$ which is symmetric on $\mathbb Z$ and uniform on each $\mathbb I_0,\cdots,\mathbb I_k$ with respective masses  $p_0,\cdots,p_k$  and $\mu_k(\Delta)=q_k$ as in Figure \ref{distrib}. Note that such sequences converge in distribution to some  symmetric probability measure $\mu$ infinitely supported in $\mathbb Z$. In addition, one can choose $p_0=0$ avoiding possibly trivial jumps. For every $k\geq 0$, introduce  an {\it i.i.d.} sequence  of random variables $(X_j^{k})_{j\geq 1}$  distributed as $\mu_k$ with the following coupling properties along $k\geq 1$: 
\begin{equation*}
X_j^{k}=X_j^{k-1}\;\mbox{ on }\; \{X_j^{k-1}\neq \Delta\}\quad \mbox{and}\quad \mathbb P(X_j^{k}\in dx | X_j^{k-1}=\Delta)=\mathbb P(\xi^kU^k+(1-\xi^k)\Delta\in dx),
\end{equation*}
where $\xi^k$ is distributed as $\mathcal B(p_{k}/q_{k-1})$, $U^k$ is uniform on $\mathbb I_k$ and $\xi^k$ and $U^k$ are independent. With these sequences of jumps, associate the so-called defective random walks denoted by $\{ W_n^k \}_{n \geq 0}$  -- they fall into the cemetery as soon as one of their jumps does -- starting from the origin. It follows from the coupling properties that $W_n^{r}=W_n^k$ for every $r\geq k$ as soon as  $W_n^{k}\neq \Delta$ in such way that we can consider the non-defective almost-sure limit random walk
\begin{equation*}
H_n:=\lim_{k\to\infty} W_n^k.
\end{equation*}
Note that the latter has for jump distribution the limit $\mu$ of the $\mu_k$ and we denote in an obvious meaning by $\{X_n\}_{n\geq 1}$ the corresponding {\it i.i.d.\@} jumps. 

In the sequel, we say for two non-negative sequences $u$ and $v$ that $u_k\preceq v_k$ if there exists $c>0$ such that $u_k\leq c v_k$ for all $k$ sufficiently large and  $u_k\asymp v_k$ whenever $u_k\preceq v_k$ and $v_k\preceq u_k$. 

\begin{lem}\label{majmin}
The following estimates holds 
\begin{equation}\label{maj}
\sqrt{\sum_{n=1}^\infty \mathbb P(W_n^{k}=0)^2} \preceq \sum_{i=1}^k\frac{1}{l_i p_i}\sqrt{\ln\left({\frac{1}{p_{i+1}}}\right)}.
\end{equation}
and
\begin{equation}\label{mino}
\sum_{n=1}^\infty\mathbb P(W_n^k=0)\mathbb P(0\leq W_n^k\leq X_{n+1}^k)\succeq \frac{y_k}{(y_k+l_k)^2}\ln\left(\frac{1}{p_{k+1}}\right).
\end{equation}
\end{lem}

\begin{proof}
 We begin with the inequality (\ref{maj}). First, we shall prove that 
\begin{equation}\label{recurence}
\mathbb P(W_n^{k}=0)\leq \mathbb P(W_n^{k-1}=0)+
\frac{1}{l_{k}}\frac{\omega(p_{k}/(1-q_{k}))}{\sqrt{n(p_{k}/(1-q_{k}))}}(1-q_{k})^n,
\end{equation}
where $\omega : (0,1)\longrightarrow (0,\infty)$ satisfies 
\begin{equation}\label{omega}
\omega(p):=\min_{u\in(0,1)}\left(\sqrt{\frac{2}{u}}+\sqrt{\frac{1}{2e p}}\frac{1}{1-u}\right)\underset{p\to 0}{\sim} \sqrt{\frac{1}{{2ep}}}.
\end{equation}
To this end, note that $\mathbb P(W_n^{k}=0)=\mathbb P(W_n^{k-1}=0)+\mathbb P(W_n^{k}=0,W_n^{k-1}=\Delta)$. Then, let $\mathbb Q$ be the conditional probability given the event $\{X_1^{k},\cdots,X_n^{k}\neq \Delta\}$ -- of $\mathbb P$-measure equal to $(1-q_{k})^n$. We can write
\begin{equation}\label{EqRecPWnk}
\mathbb P(W_n^{k}=0,W_n^{k-1}=\Delta)=(1-q_{k})^n\sum_{m=1}^n\mathbb Q(W_n^{k}=0|Z_n^{k}=m)\mathbb Q(Z_n^{k}=m),
\end{equation}
where 
\begin{equation*}
Z_n^{k}={\rm card}\left\{1\leq i\leq n : X_i^{k}\in \mathbb I_{k} \right\}.
\end{equation*}
Besides, given one of the $\binom{n}{m}$ partitions $J\sqcup I$ of $[1,n]\cap \mathbb Z$ with ${\rm card}(J)=m$, we can define
\begin{equation}
F_J:=\left(\bigcap_{j\in J} \left\{X_j^{k}\in \mathbb I_{k}\right\}\right)\cap\left(\bigcap_{i\in I}\left\{X_i^{k}\notin \mathbb I_{k} \right\}\right)\subset \{Z_n^{k}=m\}.
\end{equation}
All these  events form a partition of $\{Z_n^{k}=m\}$ itself. Under $\mathbb Q(\star |F_J)$ the random variables $X_j^{k}$ are independent. In addition, for every $j\in J$, $X_j^k$ can be written as $\theta_j Y_j$  where the $\{ Y_j \}_{j \in J}$ and the $\{ \theta_j \}$ are independent family of {\it i.i.d.\@} random variables uniformly distributed on $\mathbb I_{k}^+$ and $\{\pm 1\}$ respectively. Also, observe that under $\mathbb Q$ the random variable $Z_n^{k}$ is binomial of parameters $n$ and $p_{k}/(1-q_{k})$. Finally, lemma \ref{majmin} follows from the two technical lemmas below.


\begin{lem}\label{unif}
Let $\{ Y_j \}_{j\geq 1}$ be a sequence of independent random variables distributed uniformly on integers intervals of length $l\geq 2$. Then for every $m\geq 1$ one has 
\begin{equation}
\sup_{x\in\mathbb Z}\mathbb P(Y_1+\cdots+Y_m=x)\leq \frac{1}{l}\sqrt{\frac{2}{m}}.
\end{equation}
\end{lem}

\begin{lem}\label{binom} Let $Z$ be a random variable distributed as $\mathcal B(n,p)$. Then 
 \begin{equation}
\mathbb E\left[\frac{\mathds 1_{\left\{Z\geq 1\right\}}}{\sqrt{Z}}\right]\leq \frac{\omega(p)}{\sqrt{2np}},
\end{equation}
where $\omega : (0,1)\longrightarrow (0,\infty)$ is defined in (\ref{omega}).
\end{lem}

\begin{proof}[Proof of Lemmas \ref{unif} and \ref{binom}]
  The proof of Lemma \ref{unif} in \cite[pp. 697-698]{Rainer2007} contains some misunderstandings. To overcome these difficulties, we refer to \cite{roos}. This probabilistic estimate relies on combinatorics considerations, the so called polynomial coefficients. To go further, we allude  for instance to \cite[Section 1.16., pp. 77-78]{Comtet}.

  For Lemma  \ref{binom}, the inequality follows from the upper bound 
\begin{equation*}
\mathbb E\left[\frac{\mathds 1_{\left\{Z\geq 1\right\}}}{\sqrt{Z}}\right]\leq \frac{1}{\sqrt{npu}}\left[1+\sqrt{npu} e^{-2n(1-u)^2 p^2}\right],
\end{equation*}
obtained with a truncation argument along $\{Z\geq npu\}$ for any $u\in(0,1)$ and the Hoeffding's inequality. Since $x \exp(-x^2)\leq 1/\sqrt{2e}$ and $u\in(0,1)$ is arbitrary the result follows.
\end{proof}


After conditioning with respect to the $\theta_j$, we can apply Lemma \ref{unif} to our situation and we obtain for any $J$,  
\begin{equation*}
\mathbb Q(W_n^{k}=0|F_J)\leq \frac{1}{l_{k}}\sqrt{\frac{2}{m}},\quad\mbox{and thus}\quad 
\mathbb Q(W_n^{k}=0|Z_n^{k}=m)\leq \frac{1}{l_{k}}\sqrt{\frac{2}{m}}.
\end{equation*}
Lemma \ref{binom} and \eqref{EqRecPWnk} then imply \eqref{recurence}. Since the $p_k$ and the $q_k$ go to zero, we get for all $k$ large enough
\begin{equation}\label{recurence2}
\mathbb P(W_n^{k}=0)\leq \mathbb P(W_n^{k-1}=0)+\frac{1}{l_kp_k}{\frac{(1-q_k)^n}{\sqrt n}}.
\end{equation}
Then the upper bound \eqref{maj} follows by induction from the Minkowski inequality, the series expansion of $\ln(1+x)$ near the origin and the inequality $q_k\geq p_{k+1}$. 

It remains to prove the lower bound (\ref{mino}). 
In the sequel, we denote by $(H_n^k)_{n\geq 1}$ the random walk on $\mathbb Z$ whose common distribution of the \emph{i.i.d.} jumps, denoted by $\{ Y_n^k \}_{n \geq 0}$, is given by the conditional law $\mu_k(\star |\mathbb Z)$. The corresponding characteristic function is denoted by $\varphi_k$. It is worth noting that, under the probability measure $\mathbb Q$, $W_n^k$ is distributed as $H_n^k$  and $X_1^k$ is distributed as $Y_1^k$. If $\varphi$ stands for the characteristic function of $\mu$, then 
\begin{equation}\label{approx}
\|\varphi-\varphi_k\|_\infty\leq 2q_k\xrightarrow[k\to\infty]{} 0.
\end{equation}
Since $\varphi$ and the $\varphi_k$ are continuous and non-negative with $\varphi(0)=1$, there exists $p>0$ such that $\varphi_k(t)\geq 0$ for $k$ large enough and every $t\in [-p,p]$. Hence, Lemma \ref{lévysym} provides positive universal constants $C,L$ such that for all $k$ sufficiently large and every $\lambda\geq L/p$ and $n\geq 1$,
\begin{equation*}
\mathbb P(0\leq H_n^k\leq \lambda)\geq \frac{Q(H_n^k,\lambda)}{C}.
\end{equation*}
We will apply \cite[Theorem 1.1]{Bobkov} but before we shall give a lower and an upper estimate of the variance of $X_1^k$ given $X_1^k\in\mathbb Z$, denoted below by $\sigma_k^2$. Using \eqref{CondConst2} and the fact that $1/2\leq 1-q_k\leq 1$ for all $k$ sufficiently large, it is not difficult to see that for such $k$ 
\begin{equation}\label{variance2}
\frac{1}{3}p_k (y_k+l_k-1)^2\leq  \sigma_k^2 \leq 2 (\alpha+1)p_k(y_k+l_k)^2.
\end{equation} 
Therefore, looking at the event $\{X_n^k\geq y_k\}$, \cite[Theorem 1.1]{Bobkov} and \eqref{CondConst4} imply there exists a positive constant $\gamma$ such that for every $n\geq 1$ and $k$ sufficiently large  
\begin{equation}\label{mino1}
\mathbb P(0\leq H_n^k\leq Y_{n+1}^k)\geq \frac{\gamma p_k y_k}{\sqrt{np_k(y_k+l_k)^2}}.
\end{equation}
Since $\varphi$ is the characteristic function of an aperiodic symmetric distribution, we deduce from (\ref{approx}) that there exist $\delta\in(0,1)$ and a neighbourhood $\mathcal V$ of the origin for which $|\varphi_k(t)|\leq \delta$ for every $t\in \mathbb T^1\setminus \mathcal V$ and $k$ large enough. Lemma \ref{encadrement2} and the lower bound in  \cite{Bobkov} apply once again and we obtain from  \eqref{CondConst4} the existence of a positive constant $\theta$ such that for all $k$ large enough and every $n\geq 1$,
\begin{equation}\label{mino2}
\mathbb P(H_n^k=0)\geq \frac{\theta}{\sqrt{np_k(y_k+l_k)^2}}-\theta \delta^n.
\end{equation}
Remark that the two latter lower bounds hold for the defective random walks $W_n^k$ by adding the multiplicative term $(1-q_k)^{n+1}$. Finally, one deduce (\ref{mino}) noting that $q_k\leq p_{k+1}/(1-r)$. 
\end{proof}

Therefore, letting $k\longrightarrow \infty$ in (\ref{maj}) and (\ref{mino}) and using conditions (\ref{divergence}) and (\ref{CondConst5}), we obtain (\ref{majfinal}) by  the monotone convergence theorem. From now, one choose such a distribution $\mu$ with $p_0=0$ and consider its positive part $\nu$ normalized to be a probability. It is nothing but the distribution of $|X|$ when $X$ is distributed as $\mu$. Then one can see with the help of Corollary  \ref{coro} that the latter is the waiting times distribution announced in Theorem \ref{conjecture} but only for non-backtracking DRRW. For original DRRW we still have to work.\\

\noindent
{\bf Step 3.} We shall define the appropriate waiting times distribution for DRRW. To this end,  choose  again $\nu$ as above and let us prove it is again the desired waiting time distribution. 

We consider a geometric random variable $G$ with parameter $2/3$, a symmetric Rademacher random variable $\varepsilon$ and a  sequence of {\it i.i.d.\@} random variables $\{\tau_{k}\}_{k\geq 1}$ distributed as $\nu$ -- all these random variables are supposed independent of each others. Finally, introduce the random walk $\{H_n^\prime\}_{n\geq 0}$ whose jumps  are distributed as 
\begin{equation*}
\varepsilon\sum_{k=1}^G(-1)^{k-1}\tau_k.
\end{equation*} 
We claim that $\{H'_n\}_{n\geq 0}$ satisfies the same estimates as $\{H_n\}_{n\geq 0}$ in (\ref{majfinal}). Admitting this affirmation, it is again straightforward to deduce Theorem \ref{conjecture} for true DRRW using Corollary \ref{coro}. 

In order to complete the reasoning, look carefully at the proof of Lemma \ref{majmin} above, especially when the defective jumps $\{X_n^k\}_{n\geq 1}$ are replaced by {\it i.i.d.\@} random variables distributed as 
\begin{equation}\label{saut}
\varepsilon_{n}\sum_{i=1}^{G_n}(-1)^{i-1} \tau_{n,i}^k,
\end{equation}  
where $\{\tau_{n,i}^k\}_{n,i\geq 1}$ is an {\it i.i.d.\@} array of random variables distributed as $\nu_k(dt):=\mu_k(dt |\mathbb{N}\cup \{\Delta\})$. The other random variables are defined similarly to those appearing in Lemma \ref{majmin}, all of them are independent of each others. Conditioning with respect to $G_1,\ldots,G_n$  the upper-bound (\ref{recurence2}) still holds since 
\begin{equation*}\label{recurence3}
{\frac{(1-q_k)^{G_1+\cdots+G_n}}{\sqrt {G_1+\cdots+G_n}}}\leq 
{\frac{(1-q_k)^{n}}{\sqrt n}}\quad a.s.\mbox{.}
\end{equation*}
It turns out that the limit random walk $\{H'_n\}_{n\geq 0}$ satisfies the left-hand side of (\ref{majfinal}) and it remains to show the lower bound. 

Let us remark that the variance of the jump distribution in (\ref{saut}) still satisfies the former lower and upper bounds in (\ref{variance2}) with possibly different universal constants. The lower bound is straightforward since $G_n \geq 1$ for all $n \geq 1$ whereas for the upper bound we distinguish between even $G_n$ and odd $G_n$ and use basic conditional arguments. Furthermore, similar arguments apply to the characteristic functions since 
\begin{equation*}
\Phi_k(t)=\sum_{i=0}^\infty \frac{1}{3}\left(\frac{2}{3}\right)^{2i}|\phi_{k}(t)|^{2i}\varphi_k(t)+\sum_{i=1}^\infty \frac{1}{3}\left(\frac{2}{3}\right)^{2i-1}|\phi_{k}(t)|^{2i},
\end{equation*}
where $\varphi_k$, $\phi_k$ and $\Phi_k$ are respectively the characteristic functions of $\mu_k(dx|\mathbb Z)$, $\nu_k(dt|\mathbb N)$ and the random variable (\ref{saut}), given it does not fall into the cemetery point. It follows that the lower bound (\ref{mino1}) is true for this new distribution (we suppose $G_{n+1}=1$ into the jump $n+1$) but also the lower bound (\ref{mino2}). Obviously, the constant are possibly different. To conclude, it suffices to the take the conditional expectation with respect to $G_1,\cdots,G_{n+1}$ in the multiplicative additional terms and observe by independence and the Jensen inequality that
\begin{equation*}
\mathbb E\left[(1-q_k)^{G_1+\cdots+G_{n+1}}\right]\geq (1-q_k)^{3(n+1)/2}.
\end{equation*}
Since the additional factor $3/2$ does not change the nature of the series, we deduce that the second series of (\ref{majfinal}) is also infinite for $\{H'_n\}_{n\geq 0}$.
\end{proof}

\begin{ack}
We warmly thank Arnaud Le Ny, who introduced us to these challenging persistent random walks, for our valuable discussions and his infinite patience. We would like to thank also all people managing very important tools for french mathematicians: Mathrice which provides a large number of
services.
\end{ack}

\bibliographystyle{unsrt}
\bibliography{biblio-vnby}

\def\polhk#1{\setbox0=\hbox{#1}{\ooalign{\hidewidth
  \lower1.5ex\hbox{`}\hidewidth\crcr\unhbox0}}}
\begin{thebibliography}{10}

\bibitem{Rainer2007}
Rainer Siegmund-Schultze and Heinrich von Weizs{\"a}cker.
\newblock Level crossing probabilities. {II}. {P}olygonal recurrence of
  multidimensional random walks.
\newblock {\em Adv. Math.}, 208(2):680--698, 2007.

\bibitem{Mauldin1996}
R.~Daniel Mauldin, Michael Monticino, and Heinrich von Weizs{\"a}cker.
\newblock Directionally reinforced random walks.
\newblock {\em Adv. Math.}, 117(2):239--252, 1996.

\bibitem{PRWI}
Peggy C{\'e}nac, Arnaud Le~Ny, Basile de~Loynes, and Yoann Offret.
\newblock Persistent {R}andom {W}alks. {I}. {R}ecurrence {V}ersus {T}ransience.
\newblock {\em Journal of Theoretical Probability}, pages 1--12, 2016.

\bibitem{ccpp}
Peggy C{\'e}nac, Brigitte Chauvin, Fr{\'e}d{\'e}ric Paccaut, and Nicolas
  Pouyanne.
\newblock Context trees, variable length {M}arkov chains and dynamical sources.
\newblock In {\em S\'eminaire de {P}robabilit\'es {XLIV}}, volume 2046 of {\em
  Lecture Notes in Math.}, pages 1--39. Springer, Heidelberg, 2012.

\bibitem{Alsmeyer}
Gerold Alsmeyer.
\newblock Recurrence theorems for {M}arkov random walks.
\newblock {\em Probab. Math. Statist.}, 21(1, Acta Univ. Wratislav. No.
  2298):123--134, 2001.

\bibitem{Barbu2008}
Vlad~Stefan Barbu and Nikolaos Limnios.
\newblock {\em Semi-{M}arkov chains and hidden semi-{M}arkov models toward
  applications}, volume 191 of {\em Lecture Notes in Statistics}.
\newblock Springer, New York, 2008.
\newblock Their use in reliability and DNA analysis.

\bibitem{Magdziarz}
M.~Magdziarz, H.~P. Scheffler, P.~Straka, and P.~Zebrowski.
\newblock Limit theorems and governing equations for {L}\'evy walks.
\newblock {\em Stochastic Process. Appl.}, 125(11):4021--4038, 2015.

\bibitem{Meer1}
Peter Becker-Kern, Mark~M. Meerschaert, and Hans-Peter Scheffler.
\newblock Limit theorems for coupled continuous time random walks.
\newblock {\em Ann. Probab.}, 32(1B):730--756, 2004.

\bibitem{Meer2}
Mark~M. Meerschaert and Hans-Peter Scheffler.
\newblock Limit theorems for continuous-time random walks with infinite mean
  waiting times.
\newblock {\em J. Appl. Probab.}, 41(3):623--638, 2004.

\bibitem{MeerStra}
Mark~M. Meerschaert and Peter Straka.
\newblock Semi-{M}arkov approach to continuous time random walk limit
  processes.
\newblock {\em Ann. Probab.}, 42(4):1699--1723, 2014.

\bibitem{Straka}
P.~Straka and B.~I. Henry.
\newblock Lagging and leading coupled continuous time random walks, renewal
  times and their joint limits.
\newblock {\em Stochastic Process. Appl.}, 121(2):324--336, 2011.

\bibitem{CCPP3}
Peggy C{\'e}nac, Brigitte Chauvin, Frédéric Paccaut, and Nicolas Pouyanne.
\newblock Stationary measures for {V}ariable {L}ength {M}arkov {C}hains:
  towards a necessary and sufficient condition.
\newblock {\em Forthcoming}, 2018.

\bibitem{peggy}
P.~C{\'e}nac, B.~Chauvin, S.~Herrmann, and P.~Vallois.
\newblock Persistent random walks, variable length {M}arkov chains and
  piecewise deterministic {M}arkov processes.
\newblock {\em Markov Process. Related Fields}, 19(1):1--50, 2013.

\bibitem{Kal}
Olav Kallenberg.
\newblock {\em Foundations of modern probability}.
\newblock Probability and its Applications (New York). Springer-Verlag, New
  York, second edition, 2002.

\bibitem{berbeeRecTrans}
Henry Berbee.
\newblock Recurrence and transience for random walks with stationary
  increments.
\newblock {\em Z. Wahrsch. Verw. Gebiete}, 56(4):531--536, 1981.

\bibitem{berbeethesis}
Henry C.~P. Berbee.
\newblock {\em Random walks with stationary increments and renewal theory},
  volume 112 of {\em Mathematical Centre Tracts}.
\newblock Mathematisch Centrum, Amsterdam, 1979.

\bibitem{Gui:83}
Yves Guivarc'h.
\newblock Application d'un th\'eor\`eme limite local \`a la transience et \`a
  la r\'ecurrence de marches de {M}arkov.
\newblock In {\em Th\'eorie du potentiel ({O}rsay, 1983)}, volume 1096 of {\em
  Lecture Notes in Math.}, pages 301--332. Springer, Berlin, 1984.

\bibitem{Bab:88}
M.~Babillot.
\newblock Th\'eorie du renouvellement pour des chaînes semi-markoviennes
  transientes.
\newblock {\em Ann. Inst. H. Poincar\'e Probab. Statist.}, 24(4):507--569,
  1988.

\bibitem{Uch:07}
K\^ohei Uchiyama.
\newblock Asymptotic estimates of the {G}reen functions and transition
  probabilities for {M}arkov additive processes.
\newblock {\em Electron. J. Probab.}, 12:no. 6, 138--180, 2007.

\bibitem{guilepagestable}
Y.~Guivarc'h and Emile Le~Page.
\newblock On spectral properties of a family of transfer operators and
  convergence to stable laws for affine random walks.
\newblock {\em Ergodic Theory Dynam. Systems}, 28(2):423--446, 2008.

\bibitem{HervePen}
Loïc Herv\'e and Fran\c{c}oise P\`ene.
\newblock The {N}agaev-{G}uivarc'h method via the {K}eller-{L}iverani theorem.
\newblock {\em Bull. Soc. Math. France}, 138(3):415--489, 2010.

\bibitem{KonMeyn}
I.~Kontoyiannis and S.~P. Meyn.
\newblock Geometric ergodicity and the spectral gap of non-reversible {M}arkov
  chains.
\newblock {\em Probab. Theory Related Fields}, 154(1-2):327--339, 2012.

\bibitem{KonMeyn2}
I.~Kontoyiannis and S.~P. Meyn.
\newblock Spectral theory and limit theorems for geometrically ergodic {M}arkov
  processes.
\newblock {\em Ann. Appl. Probab.}, 13(1):304--362, 2003.

\bibitem{MeynTweedie}
Sean Meyn and Richard~L. Tweedie.
\newblock {\em Markov chains and stochastic stability}.
\newblock Cambridge University Press, Cambridge, second edition, 2009.
\newblock With a prologue by Peter W. Glynn.

\bibitem{GarRosen}
Gareth~O. Roberts and Jeffrey~S. Rosenthal.
\newblock Geometric ergodicity and hybrid {M}arkov chains.
\newblock {\em Electron. Comm. Probab.}, 2:no.\ 2, 13--25, 1997.

\bibitem{HerLedoux}
Loïc Herv\'e and James Ledoux.
\newblock Spectral analysis of {M}arkov kernels and application to the
  convergence rate of discrete random walks.
\newblock {\em Adv. in Appl. Probab.}, 46(4):1036--1058, 2014.

\bibitem{Hennion}
Hubert Hennion and Loïc Herv\'e.
\newblock {\em Limit theorems for {M}arkov chains and stochastic properties of
  dynamical systems by quasi-compactness}, volume 1766 of {\em Lecture Notes in
  Mathematics}.
\newblock Springer-Verlag, Berlin, 2001.

\bibitem{guibourg:hal}
Denis Guibourg, Lo{\"i}c Herv{\'e}, and James Ledoux.
\newblock {Quasi-compactness of Markov kernels on weighted-supremum spaces and
  geometrical ergodicity}.
\newblock 45 pages, February 2012.

\bibitem{Lunardi1}
Alessandra Lunardi.
\newblock {\em Analytic semigroups and optimal regularity in parabolic
  problems}.
\newblock Modern Birkh\"auser Classics. Birkh\"auser/Springer Basel AG, Basel,
  1995.
\newblock [2013 reprint of the 1995 original] [MR1329547].

\bibitem{CGT}
Celso Mart\'\i~nez Carracedo and Miguel Sanz~Alix.
\newblock {\em The theory of fractional powers of operators}, volume 187 of
  {\em North-Holland Mathematics Studies}.
\newblock North-Holland Publishing Co., Amsterdam, 2001.

\bibitem{Kato}
Tosio Kato.
\newblock {\em Perturbation theory for linear operators}.
\newblock Die Grundlehren der mathematischen Wissenschaften, Band 132.
  Springer-Verlag New York, Inc., New York, 1966.

\bibitem{HervPen}
Loïc Herv\'e and Fran\c{c}oise P\`ene.
\newblock On the recurrence set of planar {M}arkov random walks.
\newblock {\em J. Theoret. Probab.}, 26(1):169--197, 2013.

\bibitem{Shepp}
L.~A. Shepp.
\newblock Recurrent random walks with arbitrarily large steps.
\newblock {\em Bull. Amer. Math. Soc.}, 70:540--542, 1964.

\bibitem{Grey}
D.~R. Grey.
\newblock Persistent random walks may have arbitrarily large tails.
\newblock {\em Adv. in Appl. Probab.}, 21(1):229--230, 1989.

\bibitem{Esseen1}
C.~G. Esseen.
\newblock On the {K}olmogorov-{R}ogozin inequality for the concentration
  function.
\newblock {\em Z. Wahrscheinlichkeitstheorie und Verw. Gebiete}, 5:210--216,
  1966.

\bibitem{Esseen2}
C.~G. Esseen.
\newblock On the concentration function of a sum of independent random
  variables.
\newblock {\em Z. Wahrscheinlichkeitstheorie und Verw. Gebiete}, 9:290--308,
  1968.

\bibitem{Raugi}
Albert Raugi.
\newblock D\'epassement des sommes partielles de v.a.r.\ ind\'ependantes
  \'equidistribu\'ees sans moment d'ordre 1.
\newblock {\em Ann. Fac. Sci. Toulouse Math. (6)}, 9(4):723--734, 2000.

\bibitem{Petrov}
V.~V. Petrov.
\newblock {\em Sums of independent random variables}.
\newblock Springer-Verlag, New York-Heidelberg, 1975.
\newblock Translated from the Russian by A. A. Brown, Ergebnisse der Mathematik
  und ihrer Grenzgebiete, Band 82.

\bibitem{roos}
Lutz Mattner and Bero Roos.
\newblock Maximal probabilities of convolution powers of discrete uniform
  distributions.
\newblock {\em Statist. Probab. Lett.}, 78(17):2992--2996, 2008.

\bibitem{Comtet}
Louis Comtet.
\newblock {\em Advanced combinatorics}.
\newblock D. Reidel Publishing Co., Dordrecht, enlarged edition, 1974.
\newblock The art of finite and infinite expansions.

\bibitem{Bobkov}
Sergey~G. Bobkov and Gennadiy~P. Chistyakov.
\newblock On concentration functions of random variables.
\newblock {\em J. Theoret. Probab.}, 28(3):976--988, 2015.

\end{thebibliography}

\end{document}